\documentclass{amsart}
\usepackage{latexsym,amsxtra,amscd,ifthen}
\usepackage{amsfonts}
\usepackage{verbatim}
\usepackage{amsmath}
\usepackage{amsthm}
\usepackage{amssymb}

\numberwithin{equation}{section}

\theoremstyle{plain}
\newtheorem{theorem}{Theorem}[section]
\newtheorem{lemma}[theorem]{Lemma}

\newtheorem{corollary}[theorem]{Corollary}

\theoremstyle{definition}
\newtheorem{definition}[theorem]{Definition}
\newtheorem{example}[theorem]{Example}

\newtheorem{remark}[theorem]{Remark}

\makeatletter              % This sequence of commands will
\let\c@equation\c@theorem  % incorporate equation numbering
\makeatother

\newcommand{\la}{\langle}
\newcommand{\ra}{\rangle}

\DeclareMathOperator{\gldim}{gldim}

\DeclareMathOperator{\Ext}{Ext}

\DeclareMathOperator{\pdim}{pdim}

\DeclareMathOperator{\GKdim}{GKdim}

\DeclareMathOperator{\im}{im}

\begin{document}

\title{Artin-Schelter Regular Algebras, Subalgebras, and Pushouts}

\author{Jun Zhang}

\address{Department of Mathematics, University of Texas at Arlington. }

\email{(Jun Zhang) zhangjun19@gmail.com}

\begin{abstract}
Take $A$ to be a regular quadratic algebra of global dimension three. We observe that there are examples of $A$ containing a dimension three regular cubic algebra $C$. If $B$ is another dimension three regular quadratic algebra, also containing $C$ as a subalgebra, then we can form the pushout algebra $D$ of the inclusions $i_1:C\hookrightarrow A$ and $i_2:C\hookrightarrow B$. We show that for a certain class of regular algebras $C\hookrightarrow A,B$, their pushouts $D$ are regular quadratic algebras of global dimension four. Furthermore, some of the point module structures of the dimension three algebras get passed on to the pushout algebra $D$. 
\end{abstract}

\subjclass[2000]{16S37, 16S38, 16W50}

% 16W50 Graded rings and modules
% 16S37 Quadratic and Koszul algebras
% 16S38 Rings arising from non-commutative algebraic geometry

\keywords{Artin-Schelter regular, global dimension, quadratic algebra}

\maketitle

%%%%%%%%%%%%%%%%%%%%
%%%%%%%%%%%%%%%%%%%%
\section{Introduction} \label{sec 0}

There is a beautiful classification of global dimension three Artin-Schelter regular algebras using point modules, \cite{ATV1}. There is an ongoing effort to classify Artin-Schelter regular algebras of global dimension four. Many families of regular algebras have been recently discovered in \cite{CV, LSV, LPWZ, Sk1, Sk2, SS, VV1, VV2, VVW, Va1, Va2, ZZ1, ZZ2}. In this paper we study some dimension three regular algebras with regular subalgebras, and use them to construct pushouts that are dimension four regular algebras. We also show that the regular pushout algebras thus constructed inherit the point module structures of the dimension three algebras.

Throughout we take $k$ to be an algebraically closed field of characteristic zero. By an algebra we mean a finitely presented, connected, graded, $k$-algebra, generated in degree one. Regular algebras with global dimension three are well understood, \cite{AS, ATV1, ATV2}. They have either two generators and two degree three relations, or three generators and three degree two relations.   In global dimension four, \cite[Proposition 1.4]{LPWZ} shows a regular domain can have 2, 3, or 4 generators, and listed the respective $k$-resolutions. Following their notation we identify dimension four regular algebras as type (12221), type (13431), and type (14641). An algebra of type (12221) has two generators, a cubic relation and a degree four relation. An algebra of type (13431) has three generators, two quadratic relations and two cubic relations. An algebra of type (14641) has four generators and six quadratic relations.

Take $k$-algebras $A,B,$ and $C$, with maps $\phi_1:C\rightarrow A$ and $\phi_2:C\rightarrow B$, then we can form the $k$-algebra pushout $D$ of the maps $\phi_1$ and $\phi_2$. In general, the algebra $D$ may not be regular, even if $A,B,C$ are all regular algebras, and the maps $\phi_1$ and $\phi_2$ are inclusions. For example, take $i_1:k\hookrightarrow k\la x \ra$, and $i_2:k\hookrightarrow k\la y \ra$, then the pushout of $i_1$ and $i_2$ is the free $k$-algebra $k\la x, y \ra$, which is not regular. Since an algebra of type (12221) has no quadratic relations, and only one cubic relation, its only regular subalgebras of lower global dimensions are $k$ and $k\la x\ra$. Hence an algebra of type (12221) cannot be formed as a pushout of morphisms between its subalgebras.

We have the following example of a type (14641) algebras that is a pushout of inclusion maps between its regular subalgebras.
\begin{example} \label{example1}
Define the functions $f_1, \ldots, f_6$, $g_1, g_2$ as follows:
\[ f_1 = x_3^2-x_1x_2-x_2x_1, \ \  f_2 = x_3x_1+x_1x_3 \]
\[  f_3 = x_3x_2+x_2x_3, \ \ f_4 = x_4^2-x_1x_2-x_2x_1 \] 
\[  f_5 = x_4x_1+x_1x_4, \ \  f_6 = x_4x_2+x_2x_4, \]
\[ g_1 = x_2^2x_1-x_1x_2^2, \ \ g_2 = x_2x_1^2-x_1^2x_2. \]
Next define the algebras 
\[ A = k \la x_3, x_1, x_2 \ra / (f_1, f_2, f_3), \] 
\[ B = k \la x_1, x_2, x_4 \ra / (f_4, f_5, f_6), \]
\[ C = k \la x_1, x_2 \ra / (g_1, g_2). \]
Here $C$ is a subalgebra of both $A$ and $B$, and the pushout of the inclusion maps $i_1:C\hookrightarrow A$ and $i_2:C\hookrightarrow B$ is the algebra
\[D = k \la x_3, x_1, x_2, x_4 \ra / (f_1, \ldots, f_6).\] 
The algebras $A, B,$ and $C$ are regular graded Clifford algebras and we can readily verify that $D$ is also a regular graded Clifford algebra. % Check if D is noetherian
\end{example}

Our goal of this paper is to generalize the above example. The challenge is that while dimension three regular algebras are well understood, less is known about their possible regular subalgebras. Instead of classifying all regular subalgebras, we choose to work with algebras satisfy the following relations, which show up naturally in our study. For lack of better terminology, we call such algebras as type A pushouts.
\begin{definition} type A: \ref{type A}
Ordering monomials in each degree according to $x_4>x_3>x_2>x_1$. Let $r_1, \ldots, r_6$ be quadratic relations, and $r_7, r_8$ be cubic relations, with leading term given as follows,
\begin{eqnarray*}
r_1 &=& x_3^2 - f_1(x_1,x_2,x_3) \\
r_2 &=& x_3x_1 - f_2(x_1,x_2,x_3) \\
r_3 &=& x_3x_2 - f_3(x_1,x_2,x_3) \\
r_4 &=& x_4x_1 - f_4(x_1,x_2,x_4) \\
r_5 &=& x_4x_2 - f_5(x_1,x_2,x_4) \\
r_6 &=& x_4^2 - f_6(x_1,x_2,x_4) \\
r_7 &=& x_2^2x_1 - f_7(x_1, x_2) \\
r_8 &=& x_2x_1^2 - f_8(x_1, x_2)
\end{eqnarray*}
Here $f_1, \ldots, f_6$ consist of lower order quadratic monomials in the given variables, and $f_7, f_8$ consist of lower order cubic monomials in $x_1$ and $x_2$. Let the algebras $A, B, C$ be defined as 
\[ A = k \la x_3, x_1, x_2 \ra / (r_1, r_2, r_3) \]
\[ B = k \la x_1, x_2, x_4 \ra / (r_4, r_5, r_6) \]
\[ C = k \la x_1, x_2 \ra / (r_7, r_8). \]
Further assume that $A, B, C$ are all regular of global dimension three and we have inclusions $i_1:C\hookrightarrow A$ and $i_2:C\hookrightarrow B$. Then we call the algebra $D$ formed by the pushout of $i_1$ and $i_2$ a {\it type A} algebra.
\end{definition}
We list our main results using the above definition
\begin{theorem} \label{main thm} (Theorems \ref{thm 2}, \ref{noetherian}, Corollary \ref{pmD}) 
\begin{enumerate}
\item We have a complete list of type A algebras.
\item A type A algebra with generic coefficients is Artin-Schelter regular of global dimension four. We list explicitly the generic conditions.
\item A type A algebra has enough normal elements, is Auslander-regular, Cohen-Macaulay, and is a strongly noetherian domain.
\item The point modules for a type A algebra $D$ consist of compatible point modules of algebras $A, B,$ and $C$, plus two additional point modules corresponding to the points $e_3$ and $e_4$. 
\end{enumerate}
\end{theorem}

\begin{remark}
\begin{enumerate}
\item For the remainder of the paper, unless otherwise stated, we will use $C$ to represent a dimension three regular cubic algebra, $A$ and $B$ to represent dimension three regular quadratic algebras, and $D$ to represent a regular algebra of type (14641), which we also refer to as a dimension four regular quadratic algebra.
\item The algebra $D$ from example \ref{example1} has a reflection automorphism $\sigma$ given by $\sigma(x_{i}) = x_{i}$ for $i = 1, 2, 3$, and $\sigma(x_4) = -x_4$. We have a related ongoing study of type (14641) regular algebras with such a reflection automorphism. The relations we use to define type A algebras show up more naturally under that setting.
\item Many type A algebras are also skew Clifford algebras as defined by \cite{CV}.
\end{enumerate}
\end{remark}

Here is an outline of our paper. In section \ref{sec 1} we give the definition and some basic properties of AS-regular algebras. In section \ref{subalgebra} we study regular subalgebras of regular algebras, and show that they are the natural candidate to form pushouts. In section \ref{sectionpm} we give a characterization of point modules of regular algebras that are also point modules of subalgebras. In section \ref{sectiontypeA} we define the type A algebras, show that they are a natural choice to focus our study on, and prove our main theorem, up to the existence of a potential $k$ resolution. In section \ref{algebra C} we compute all possible three dimensional cubic regular algebras $C$ in a type A pushout. In section \ref{algebra A} we compute all possible dimension three regular quadratic algebras that can appear in a type A pushout. In section \ref{algebra D} we list all type A algebras and verify that each algebra has a $k$ resolution, thus completing the proof of our main theorem.

There are many interesting questions we would like to answer:
\begin{enumerate}
\item Our current definition of type A algebra involves choosing the representation of our quadratic relations. We would like to give an alternate characterization that is basis free.
\item We would like to find all regular pushout algebras and extend this construction to higher dimensions.
\item We conjecture that any regular algebra with partial relations $r_1, \ldots, r_8$ is a type A algebra.
\item We would like to develop more techniques to analyze the properties of regular algebras, using regular subalgebras.
\end{enumerate}

%%%%%%%%%%%%%%%%%%%%%
\section{Regular Algebra} \label{sec 1}
%%%%%%%%%%%%%%%%%%%%%

Throughout $k$ is an algebraically closed base field of characteristic zero. An algebra $D$ is called {\it connected graded} if 
\[D=k\oplus D_1\oplus D_2\oplus \cdots\]
with $1\in k=D_0$ and $D_iD_j\subset D_{i+j}$ for all $i,j$. If $D$ is connected graded, then we also use $k$ to denotes the trivial graded module $D/D_{\geq 1}$. A connected graded algebra $D$ is called {\it Artin-Schelter regular}, or just {\it regular} for short, if the following three conditions hold.
\begin{enumerate}
\item[(AS1)]
$D$ has finite global dimension $d$.
\item[(AS2)]
$D$ has finite Gelfand-Kirillov dimension, i.e., there is a positive number $c$ such that $\dim D_n< c\; n^c$ for
all $n\in \mathbb{N}$.
\item[(AS3)] 
$D$ is {\it Gorenstein}, namely, there is an
integer $l$ such that,
\[ \Ext^i_D({_Dk}, D)=\begin{cases} k(l) & \text{ if }
i=d\\
                                0   & \text{ if }
i\neq d
\end{cases}
\]
where $_Dk$ is the left trivial $D$-module; the same condition holds for the right trivial $D$-module $k_D$.
\end{enumerate}
If $D$ is regular, then we refer to the global dimension of $D$ simply as the {\it dimension} of $D$. The notation $(l)$ in (AS3) is the $l$-th degree shift of graded modules.

\begin{remark}
In this paper we further assume that all graded algebras are generated in degree one. From now on, by an algebra we mean a finitely presented, connected, graded, $k$-algebra, generated in degree one.
\end{remark}

If $D$ is regular, then by \cite[Proposition 3.1.1]{SZ}, the trivial right $D$-module $k_D$ has a minimal free resolution of the form
\begin{equation} \label{E2.0.1}
0\to P_{d}\to \cdots P_{1}\to P_{0}\to k_D\to 0 \tag{E2.0.1}
\end{equation}
 where $P_{w}=\oplus_{s=1}^{n_w}D(-i_{w,s})$ for some finite integers $n_w$ and $i_{w,s}$. The Gorenstein condition (AS2) implies that the above free resolution is symmetric in the sense that the dual complex of \eqref{E2.0.1} is a free resolution of the trivial left $D$-module after a degree shift. As a consequence, we have $P_0=D$, $P_{d}=D(-l)$, $n_w=n_{d-w}$, and $i_{w,s}+ i_{d-w, n_w-s+1}=l$ for all $w,s$.

Regular algebras of dimension three have been classified by Artin, Schelter, Tate and Van den Bergh \cite{AS,ATV1,ATV2}. A regular algebra of dimension three is generated by either two or three elements. If $A$ is generated by three elements, then $k_A$ has a minimal free resolution  of the form
\[ 0\to A(-3)\to A(-2)^{\oplus 3}\to A(-1)^{\oplus 3}\to A\to k_A\to 0. \]
If $C$ is generated by two elements, then $k_C$ has a minimal free resolution of the form
\[ 0\to C(-4)\to C(-3)^{\oplus 2}\to C(-1)^{\oplus 2}\to C\to k_C\to 0. \]

If $D$ is a regular domain of dimension four, then $D$ is generated by 2, 3, or 4 elements \cite[Proposition 1.4]{LPWZ}. The following lemma is well known (for reference see \cite[Lemma 1.3 and (R2)]{ShV}). The transpose of a matrix $M$ is denoted by $M^t$.
\begin{lemma}
\label{xxlem2.1} 
Let $D$ be a regular graded domain of dimension four. Suppose $D$ is generated by elements $x_1,x_2,x_3,x_4$. Then
\begin{enumerate}
\item
$D$ is of type (14641), namely, the trivial right $D$-module $k$ has a free resolution
\begin{equation}
\label{E2.1.1}
0\to D(-4) \xrightarrow{\partial_4} D^{\oplus 4}(-3)
\xrightarrow{\partial_3} D^{\oplus 6}(-2)
\xrightarrow{\partial_2} D^{\oplus 4}(-1)
\xrightarrow{\partial_1} D
\xrightarrow{\partial_0} k_D\to 0
\tag{E2.1.1}
\end{equation}
where $D^{\oplus n}$ is the free right $D$-module written as an $n \times 1$ matrix.

\item
$\partial_0$ is the augmentation map with $\ker \partial_0=D_{\geq 1}$.
\item
$\partial_1$ is the left multiplication by
$X=[x_1,x_2,x_3,x_4]$.
\item
$\partial_2$ is the left multiplication by a $4\times 6$-matrix $R=(r_{ij})_{4\otimes 6}$ such that $r_i:=\sum_{i=1}^4
x_i r_{ij}$, for $i=1,\ldots,6$, are the 6 relations of $D$.
\item
$\partial_3$ is the left multiplication by a $6\times 4$-matrix $R'=(r'_{ij})_{6\times 4}$.
\item
$\partial_4$ is the left multiplication by $[x'_1,x'_2,x'_3,x'_4]^t$ where $\{x'_1,x'_2,x'_3,x'_4\}$ is a set of generators for $D$.
\textup{(}So each $x'_i$ is a $k$-linear combination of $\{x_i\}_{i=1}^4$.
\textup{)}
\item 
$[x_1,x_2,x_3,x_4]R=0$, $RR'=0$,
$R'[x'_1,x'_2,x'_3,x'_4]^t=0$ and its entries span the six relations $r_1, \ldots, r_6$.
\end{enumerate}
\end{lemma}

\begin{definition} Potential Resolution: \label{PoRes}  \\
Take $D$ a quadratic domain of the form
\[ D = k<x_1, x_2, x_3, x_4>/(r_1, r_2, r_3, r_4, r_5, r_6) \]
where $\{ x_1, \ldots, x_4 \}$ is a set of degree one generators and $\{ r_1, \ldots , r_6 \}$ is a set of quadratic relations. We call a $D$-complex a {\it potential resolution} if it has the same form as \eqref{E2.1.1}, with the additional requirement that $\partial_4$ is the left multiplication by $[x_1, x_2, x_3, x_4]^t$. If we take $T=(t_{ij})_{6\times 4}$ to be the $6\times4$ matrix such that $r_i = \sum_{j=4}^4t_{ij}x_j$, then the matrix $R'$ from lemma \ref{xxlem2.1} is equal to $ST$ for some $6\times6$ invertible matrix $S$ with entries from $k$. Thus a potential resolution has the form
\begin{equation} \label{*}
0 \rightarrow D(-4) \buildrel{X^t} \over \longrightarrow  D(-3)^{\oplus 4} \buildrel{ST} \over \longrightarrow   D(-2)^{\oplus 6}  \buildrel{R} \over \longrightarrow   D(-1)^{\oplus 4}  \buildrel{X} \over \longrightarrow  D  \rightarrow  k_D  \rightarrow 0,
\end{equation}
where $X = [x_1, x_2, x_3, x_4]$, $XR = [r_1, r_2, r_3, r_4, r_5, r_6]$, $T X^t = [r_1, r_2, r_3, r_4, r_5, r_6]^t$, and $RST=0$.
\end{definition}

We note our definition of a potential resolution is consistent with \cite{AS}, except in dimension four, and using right modules instead of left modules. We list the follow results concerning potential resolutions for later use. The Hilbert function of an algebra $D$ are the numbers $d_n = \dim_k D_n$, and the Hilbert series is defined to be $H_D(t) = \sum_n d_nt^n$. An algebra $D$ of type (14641) has Hilbert series $H_D(t) = (1-t)^{-4}$, \cite{LPWZ}.

\begin{lemma} \label{lem 1}
Assume $D$ has a potential resolution and Hilbert Series $H_D(t) = (1-t)^{-4}$. If in addition the partial complex 
\begin{equation} \label{**}
0 \rightarrow D(-4) \buildrel{X^t} \over \longrightarrow D(-3)^{\oplus 4} \buildrel{T} \over \longrightarrow D(-2)^{\oplus 6} \rightarrow \cdots 
\end{equation}
is exact, then the potential resolution (\ref{*}) is exact and the algebra $D$ has global dimension four (AS1).
\end{lemma}
\begin{proof}
Since $S$ is invertible, the exactness of the complex (\ref{**}) implies the potential resolution (\ref{*}) is exact at the $D(-3)^{\oplus 4}$ and $D(-4)$ positions. That (\ref{*}) is exact at the $k_D$ and $D$ positions is clear. A result of Govorov \cite{Go} gives the exactness of (\ref{*}) at the $D(-1)^{\oplus 4}$ position. Lastly $H_D(t) = (1-t)^{-4}$ and exactness at all other positions of (\ref{*}) give us exactness at the $D(-2)^{\oplus 6}$ position.\\
Thus the potential resolution (\ref{*}) is exact and $\gldim(D) = \pdim(k_D) = 4$.
\end{proof}

Since the ring $D^{op}$ has the same Hilbert series as $D$, we can apply lemma \ref{lem 1} to the dual complex of (\ref{*}) to obtain the following corollary
\begin{corollary} \label{cor 1}
Assume that $D$ satisfies the conditions of lemma \ref{lem 1}, and the following partial complex is exact. 
\begin{eqnarray} \label{*****}
&& 0 \rightarrow D^{op}(-4) \buildrel{X^t} \over \longrightarrow D^{op}(-3)^{\oplus 4} \buildrel{R^t} \over \longrightarrow D^{op}(-2)^{\oplus 6} \rightarrow \cdots.
\end{eqnarray}
Then the dual complex of (\ref{*}) is also exact and $D$ satisfies (AS3). Hence $D$ is regular. In particular, if $D$ satisfies the conditions of lemma \ref{lem 1}, and $D \cong D^{op}$, then $D$ is regular.
\end{corollary}

% If D is a regular algebra, then D^{op} is also a regular algebra.
% We have in general GKdim D = GKdim D^{op} and H_D(t) = H_D^{op}(t).
% Are there any dimension three algebras with D not isomorphic to D^{op}?
% In general, are D^{op} always isomorphic to D for regular algebras? 

We state the following result on growth rates of algebras for later usage
\begin{lemma} \label{growth}
Let $D$ be a quadratic algebra of the form
\[ D = k<x_1, x_2, x_3, x_4>/(r_1, r_2, r_3, r_4, r_5, r_6) \]
with relations $r_m$ given by
\[ r_m = \sum_{ij}a_{ij}^mx_ix_j \]
for $i,j = 1, \ldots, 4$ and $m = 1, \ldots, 6$.

If we have $a_{44}^m=a_{43}^m=a_{34}^m = 0$ for all $m$, then $\GKdim D = \infty $.
\end{lemma}

\begin{proof} 
Let $n_1, \ldots, n_{s-1}$ be a sequence of positive integers, and $n_0, n_s$ two non-negative integers. We show that the $y$ monomials, defined as 
\[ \{y_{n_0,n_1,\ldots,n_s} = x_4^{n_0}x_3x_4^{n_1}x_3x_4^{n_2}\cdots x_3 x_4^{n_s} \}\]
are linearly independent, by showing that these monomials do not appear in any relation of $A$. 

Assume that there are relations in $A$ with at least one of the above $y$ monomials. Take $g$ to be such a relation, with minimal degree $n$. By the forms of our generating relations $r_i$, we have $n>2$. We write
\[ g = \sum_{i=1}^4 (x_if_i + f_{i+4}x_i) \]
Here the functions $f_1, \ldots, f_8$ are relations in $A$ of degree $n-1$. It is easy to see then that one of the functions $f_3, f_4, f_7, f_8$ must contain a $y$ monomial. This contradicts the assumption that $g$ has minimal degree among all relations containing $y$ monomials.

Thus none of the $y$ monomials can appear in a relation of $A$, so they are linearly independent. The number of $y$ monomials in degree $n$ is the same as the number of sequences $(n_0,n_1,\ldots,n_s)$, with $s+\sum n_i = n$. It is easy to see then the number of $y$ monomials has faster than polynomial growth. So $\GKdim D = \infty$.
\end{proof}

%%%%%%%%%%%%%%%%%%%%
%%%%%%%%%%%%%%%%%%%%
\section{Subalgebras and Pushouts} \label{subalgebra}

Take $k$-algebras $A,B,$ and $C$, with morphisms $\phi_1:C\rightarrow A$ and $\phi_2:C\rightarrow B$. Write
\[ C = k \la x_c \ra / (f_{\gamma}) \]
\[ A = k \la y_a \ra / (g_{\alpha}) \]
\[ B = k \la z_b \ra / (h_{\beta}) \]
Then the pushout of the maps $\phi_1$ and $\phi_2$ is the $k$-algebra
\[ D = k \la x_c, y_a, z_b \ra / (f_{\gamma}, g_{\alpha}, h_{\beta}, \phi_1(x_i)-x_i, \phi_2(x_i)-x_i ) \]
It is straightforward to check that the algebra $D$ satisfies the universal property of pushout.

We feel an analysis of regular pushouts allowing general morphisms $\phi_1$ and $\phi_2$ is too big a task at this moment. Since a regular algebra has finite GK-dimension, we have
\begin{enumerate}
\item If $\phi_1$ is injective, then $\GKdim C \le \GKdim A$. In this case we can identify the algebra $C$ with its image $\phi_1(C)$ and view $C$ as a subalgebra of $A$.
\item If $\phi_1$ is surjective, then $\GKdim C \ge \GKdim A$. In this case we can identify the algebra $A$ with a quotient algebra of $C$.
\end{enumerate}

While it is not known in general for regular algebra $D$ if $\dim D = \GKdim D$, it is true that for dimension three algebras, their GK-dimension is also three, \cite{ATV1}. Our current goal is to construct dimension four regular algebras, using dimension three algebras. Thus it is natural for us to start with $\phi_1$ and $\phi_2$ injective. From now on, we restrict our attention to the cases where $C$ is a subalgebra of $A$ and $B$, and $D$ is the pushout of the inclusion maps $i_1:C\hookrightarrow A$ and $i_2:C\hookrightarrow B$. We note that the algebras $A$ and $B$ may not be subalgebras of the pushout $D$.

%%%%%
\begin{comment}
For $\phi_1$ and $\phi_2$ surjective, consider
\[ C = k [x_1, x_2, x_3] \]
\[ A = k [x_1, x_2] \]
\[ B = k [x_1, x_3] \]
\[ D = k [x_1, x_2, x_3 \]
Uninteresting, what would be some interesting application?
\end{comment}
%%%%%

Our main interest is in constructing algebras of type (14641), which have four generators and six quadratic relations, using lower dimensional regular algebras. By listing out all such possible pushouts, we can check that the only case where the pushout has four generators and six quadratic relations is when $A$ and $B$ are both dimension three regular quadratic algebras, and $C$ is a dimension three regular cubic algebra. We have
\[ C = k \la x_1, x_2 \ra / I_c=(g_1, g_2) \]
\[ A = k \la x_1, x_2, x_3 \ra / I_A=(f_1, f_2, f_3) \]
\[ B = k \la x_1, x_2, x_4 \ra / I_B=(f_4, f_5, f_6) \]
Here $f_1,f_2,f_3$ are quadratic functions in $x_1,x_2,x_3$; $f_4,f_5,f_6$ are quadratic functions in $x_1,x_2,x_4$; $g_1, g_2$ are cubic functions in $x_1$ and $x_2$. Since $C$ is a subalgebra of $A$ and $B$, and there is no quadratic relation in $C$ between $x_1$ and $x_2$, we have the functions $f_i$'s can not be relations between $x_1$ and $x_2$ only. In this case the pushout of the inclusion maps is the algebra $D$ given by
\[ D = k \la x_1, x_2, x_3, x_4 \ra / I_D=(f_1, \ldots, f_6) \]
We observe that there are no $x_3x_4$ or $x_4x_3$ term in the defining relations of $D$. If in addition there are no $x_4^2$ terms in the relations $f_4, f_5, f_6$, then by lemma \ref{growth}, $\GKdim D = \infty$, and $D$ can not be regular. Thus we require the algebra $B$ to have some non-trivial $x_4^2$ terms in its defining relations. We can choose to have $f_4 = x_4^2 + f_4'$. Then after linear combination of defining relations, we can choose to have no $x_4^2$ terms in $f_5$ and $f_6$. The same reasoning works with $x_3^2$ terms for algebra $A$.

\begin{definition} \label{defpushout}
Let the algebras $A$ and $B$ be two dimension three regular quadratic algebras, and the algebra $C$ a common dimension three regular cubic subalgebra of both $A$ and $B$, with generators and relations given as
\[ C = k \la x_1, x_2 \ra / I_c=(g_1, g_2) \]
\[ A = k \la x_1, x_2, x_3 \ra / I_A=(x_3^2+f_1, f_2, f_3) \]
\[ B = k \la x_1, x_2, x_4 \ra / I_B=(x_4^2+f_4, f_5, f_6). \]
Here $g_1, g_2$ are cubic functions in $x_1$ and $x_2$; $f_1, f_2, f_3$ are quadratic functions in $x_1,x_2x_3$, without any $x_3^2$ term; $f_4,f_5,f_6$ are quadratic functions in $x_1,x_2,x_4$, without any $x_4^2$ term.

From now on, by a {\it pushout algebra}, we mean the $k$-algebra pushout of the two inclusion maps $i_1:C\hookrightarrow A$ and $i_2:C\hookrightarrow B$. We denote the pushout algebra $D$ as $D = A\cup_CB$. The algebra $D$ is the quadratic algebra with relations
\[ D = k \la x_1, x_2, x_3,x_4 \ra / I_D=(x_3^2+f_1, f_2, f_3, x_4^2+f_4, f_5, f_6). \]
\end{definition}

We note that the algebras $A$ and $B$ are not necessarily subalgebras of $D$. While the algebras $A,B,C$ are all noetherian domains by the works of \cite{AS, ATV1}, the pushout $D$ is not assumed to be regular, noetherian, or a domain. If the algebra $D$ is regular, then its Hilbert series must be $H_D(t) = (1-t)^{-4}$, so we have $\dim_k D_3 = 20$. We have the following result.

\begin{lemma} \label{degree3}
If the algebra $D = A\cup_C B$ satisfies $\dim_k D_3 = 20$, then $A_3 \hookrightarrow D_3$ and $B_3 \hookrightarrow D_3$. 
\end{lemma}

\begin{proof}
In degree two, it is clear that $A_2\hookrightarrow D_2$, $B_2\hookrightarrow D_2$, and $\dim_k D_2 = 6$. In degree three we have $\dim_k A_3=\dim_k B_3 =10$ and $\dim_k C_3 = 6$, by \cite[equation 1.15]{AS}. Since $\dim_k D_3 = 20$. We can count the number of degree three relations in $D$ to be $\dim_k (I_D)_3 = 4^3 - 20 = 44$. On the other hand, we have 
\begin{eqnarray*}
(I_D)_3 &=& span\{x_{1,2,3,4}\} (I_D)_2 + (I_D)_2 span\{x_{1,2,3,4}\} \\
 &=& span\{x_{1,2,3,4}\} ((I_A)_2 +(I_B)_2) + ((I_A)_2 +(I_B)_2) span\{x_{1,2,3,4}\} \\
 &=& (I_A)_3 + (I_B)_3 + x_4(I_A)_2 + x_3(I_B)_2 + (I_A)_2x_4 + (I_B)_2x_3
\end{eqnarray*}
We note the vector spaces appear in the above sum may have intersections. In particular, since $C$ is a common subalgebra of both $A$ and $B$, we have
\begin{eqnarray*}
\dim_k ((I_A)_3 + (I_B)_3 ) &\le& \dim_k (I_A)_3 + \dim_k (I_B)_3 - \dim_k (I_C)_3  \\
 &=& (3^3 - 10) + (3^3- 10) - (2^3-6) \\
 &=& 32
\end{eqnarray*}
Since each of $x_4(I_A)_2$, $x_3(I_B)_2$, $(I_A)_2x_4$, $(I_B)_2x_3$ is a three dimensional vector space, and $\dim_k (I_D)_3 = 44$, we conclude that 
\[ \dim_k ((I_A)_3 + (I_B)_3 ) = 32, \] 
\[ \dim_k (x_4(I_A)_2+x_3(I_B)_2+(I_A)_2x_4+(I_B)_2x_3) =12,\] 
and their intersection is 0.

By symmetry, it is sufficient to show that $A_3 \hookrightarrow D_3$. This is the same as showing $D$ has no additional degree three relations between $x_1, x_2, x_3$, except those already in $A$. Based on the dimension count we just did, this is equivalent to show that the vector space $x_4(I_A)_2+x_3(I_B)_2+(I_A)_2x_4+(I_B)_2x_3$ does not contain a function of only $x_1, x_2, x_3$. To show this, we list out the generating relations of $D$ as
\[ r^m = \sum_{ij} k_{ij}^mx_ix_j \]
Here $1\le i,j \le 4$ and $1 \le m \le 6$. We note $k_{33}^1=k_{44}^4=1$, and many of the other coefficients are 0. We then write down a general element in the space $x_4(I_A)_2+x_3(I_B)_2+(I_A)_2x_4+(I_B)_2x_3$. After collecting terms, we set any coefficient of terms containing $x_4$ as 0. This gives us a system of equations in the coefficients $k_{ij}^m$. We do not list out all the equations. Since we have $k_{33}^1=1$, it is easy to see this system of coefficients has no solution, thus there are no additional relations in $D_3$ between $x_1, x_2, x_3$. We conclude that $A_3 \hookrightarrow D_3$. By symmetry, we also have $B_3 \hookrightarrow D_3$.
\end{proof}

\begin{comment}
In higher dimension, choose a vector space $J \subset (I_D)_n$ satisfying $J \cup (I_A)_n = J \cup (I_B)_n = 0$, and $(I_D)_n = (I_A)_n +(I_B)_n + J$. We have 
\begin{eqnarray*}
(I_D)_{n+1} &=& x_{1,2,3,4}(I_D)_n + (I_D)_nx_{1,2,3,4} \\
 &=& x_{1,2,3,4}((I_A)_n +(I_B)_n + J) + ((I_A)_n +(I_B)_n + J)x_{1,2,3,4} \\
 &=& (I_A)_{n+1} + (I_B)_{n+1} + x_4(I_A)_n + x_3(I_B)_n + (I_A)_nx_4 + (I_B)_nx_3 + x_{1,2,3,4}J + Jx_{1,2,3,4}
\end{eqnarray*}
The problem is $ x_{1,2,3,4}J, Jx_{1,2,3,4}$, since $x_4(x_4...  )$ may reduce into $I_B$. 
\end{comment}

\begin{remark}
The above lemma does not readily generalize to relations of degree four or higher for general pushout algebras. In section \ref{sectiontypeA} we identify a subclass of pushout algebras, which we call type A algebras. We will show that for type A algebras we have $A \hookrightarrow D$ and $B \hookrightarrow D$.
\end{remark}

\begin{remark}
While not considered in this paper, we list the following possibilities: Take $C$ a regular algebra of dimension two, with
\[ C = k \la x_1, x_2 \ra /(f_1) \]
\[ A = k \la x_1, x_2, x_3 \ra / (f_1, f_2, f_3) \]
\[ B = k \la x_1, x_2, x_4 \ra / (f_1, f_4, f_5) \]
Here $f_1,\dots,f_5$ are quadratic relations. In this case the pushout $D$ is generated by five quadratic relations
\[ D = k \la x_1, x_2, x_3, x_4 \ra / (f_1, \ldots, f_5), \]
so $D$ itself can not be a regular algebra. 
\begin{enumerate}
\item If $D$ is not a domain, let $D'$ be the ``largest'' domain with $D \rightarrow D'$ being a surjection, if one exist. The algebra $D'$ may be a quadratic algebra with six relations, and hence a candidate for a type (14641) algebra. We do not have an example of this case yet.
\item Alternatively, we can consider all algebras $D'' = D/(f_6)$ for some quadratic relation $f_6$. We note in this case we have the example
\[ C = k \la x_1, x_2 \ra /(x_2x_1-x_1x_2) \]
\[ A = k \la x_1, x_2, x_3 \ra / (x_2x_1-x_1x_2, x_3x_1-x_1x_3, x_3x_2-x_2x_3) \]
\[ B = k \la x_1, x_2, x_4 \ra / (x_2x_1-x_1x_2, x_4x_1-x_1x_4, x_4x_2-x_2x_4) \]
If we pick $f_6 = x_4x_3-x_3x_4$, then $D/(f_6)$ is the commutative algebra in four variables, hence regular.
\end{enumerate}
\end{remark}

%%%%%%%%%%%%%%%%%%%%
%%%%%%%%%%%%%%%%%%%%
\section{Point Modules of Regular Subalgebras} \label{sectionpm}

Take $A$ a graded $k$ algebra, generated in degree one. A right $A$-module $M$ is a right point module of $A$ if $M_0 = k$, $M = M_0 A$, and $\dim_k M_i = 1$ for $i \ge 0$. A left point module is similarly defined. Unless otherwise stated, a point module to us will be a right point module. Since each $M_i$ is a one dimensional $k$ vector space, we can pick generators $m_i$ for $M_i$. Choose generators $x_1, \ldots, x_n$ for $A_1$, and define numbers $a^i_j$ by $m_{i-1}x_j = m_i a^i_j $. The module structure of $M$ is then determined by the sequence of points $(a^1, a^2, \ldots )$ where $a^i = (a^i_1, a^i_2, \ldots, a^i_n) \in \mathbb{P}^{n-1}(k)$. We note the sequence $(a^1, a^2, \ldots )$ must be compatible with the relations of $A$, i.e. $(a^1, a^2, \ldots )$ must be in the zero locus of $I$, the generating ideal of $A$. For dimension three regular algebras, we have the following result by \cite{ATV1}.

\begin{theorem} \label{pmdim3}
Take $A$ a dimension three regular quadratic algebra. A point module of $A$ is determined by its first coordinate $a^1 \in \mathbb{P}^2$. All point modules of $A$ form a variety $E$ in $\mathbb{P}^2$. There is an automorphism $\sigma$ of $E$ which generates the sequence $(a^1, a^2, \ldots )$ by $a^n = \sigma(a^{n-1})$.

Take $C$ a dimension three regular cubic algebra. A point module of $C$ is determined by its first two coordinates $(c^1, c^2) \in \mathbb{P}^1\times \mathbb{P}^1$. All point modules of $C$ form a variety $E$ in $\mathbb{P}^1\times \mathbb{P}^1$. There is an automorphism $\tau$ of $E$ which generates the sequence $(c^1, c^2, \ldots )$ by $c^n = \tau( c^{n-1}, c^{n-2})$.
\end{theorem}

Suppose $C$ and $A$ are two regular algebras, with a morphism $\phi: C \hookrightarrow A$. Take $M$ a point module of $A$. Then $M$ has a $C$ module structure. We note that viewing $M$ as a $C$ module, we still have $\dim_k M_i = 1$, but $M$ may not be a point module of $C$, since $M_0C$ may not equal to $M$. 

\begin{definition} \label{subpm}
Suppose $C$ and $A$ are two regular algebras, with a map $\phi: C \rightarrow A$, and $M$ a point module of $A$. If $CM_0 = M$, then $M$ is also a point module of $C$, and we call $M$ an {\it $(A,C)$ point module }. Otherwise we say $M$ is an $(A,notC)$ module. 
\end{definition}

Note in the above definition, we do not require $C$ to be a subalgebra of $A$. In general, we have $M$ is an $(A,notC)$ module if there is some $i$ with $C_iM_0 = 0$, and this condition is equivalent to have some $i$, with $C_1M_i = 0$. 

While some of our following works have generalizations, they are easier to state for pushouts. To keep notations simple, for the rest of this section, we assume the algebras $A, B, C, D$ are as described in definition \ref{defpushout}. Then the point modules of $A$ correspond to points in $\mathbb{P}^2$ by theorem \ref{pmdim3}. Let $M$ be a point module of $A$ corresponding to the point $(0,0,1) \in \mathbb{P}^2$, with generators $m_i$ for $M_i$. Then we have $m_0x_1 = m_0x_2 = 0$ and $m_0x_3 = m_1$ for $m_1$. Since $C_1$ is spanned by $x_1$ and $x_2$, we have $M_0C_1 = 0$, so $M$ is an $(A,notC)$ module. The following lemma shows that the point module $(0,0,1)$ generates all $(A,notC)$ modules in this situation. 

\begin{lemma} \label{AnotC}
Suppose $M$ is an $(A,notC)$ module, and $\sigma$ is a map given by theorem \ref{pmdim3}. Then $M$ corresponds to $\sigma^{-n}(0,0,1)$ for some $n \ge 0$. As a consequence, if $A$ has no point module corresponding to $(0,0,1)$, then all point modules of $A$ are $(A,C)$ modules.
\end{lemma}
\begin{proof}
For each $i \ge 1$, we have
\[ C_1 M_i \subset A_1 M_i = M_{i+1} \cong k \]
So either $C_1M_i = M_{i+1}$, or $C_1M_i = 0$. Since $M$ is an $(A,notC)$ module, there is a first number $n$ with $C_1M_n = 0$. By our identification of $M$ with a sequence $(a^1, a^2, \ldots )$ of points in $\mathbb{P}^2$, having $C_1M_n=0$ is the same as having $a^n = (0,0,1)$. By theorem \ref{pmdim3}, $a^n = \sigma^n(a^1)$. So $a^1 = \sigma^{-n}(0,0,1)$.
\end{proof}

\begin{definition} \label{compatiblepm}
Take $D=A\cup_CB$. Let $M$ be a point module of $A$, and $M'$ a point module of $B$. We say $M$ and $M'$ are a pair of {\it compatible point modules} of $A$ and $B$ over $C$, if $M \cong M'$ as $C$ modules. 
\end{definition} 

We note that for $M$ and $M'$ to be compatible, it is not required for $M$ to be an $(A,C)$ module, or $M'$ to be a $(B,C)$ module. If $M$ is an $(A,C)$ module, then any point module of $B$ compatible with it must be a $(B,C)$ module. If we have $M$ an $(A,C)$ module, and $M'$ a $(B,C)$ module, then $M$ and $M'$ are compatible if $M_0C \cong M'_0C$. If we start with a point module $N$ of $D$ that is both a $(D,A)$ module and a $(D,B)$ module, then $N$ viewed as an $A$ module, and $N$ viewed as a $B$ module is a pair of compatible point modules. We have the following converse.
\begin{lemma}
Take $D=A\cup_CB$. Let $M$ and $M'$ be a pair of compatible point modules of $A$ and $B$ over $C$. Then $M$ has a natural $D$ module structure satisfying $M \cong M'$ as $B$ modules. Furthermore, under this $D$ module structure, $M$ is a point module of $D$.
\end{lemma}

\begin{proof}
Let $\phi$ be an $C$ module isomorphism between $M$ and $M'$. We choose in each degree $i$ a generator $m_i$ for $M_i$. Then $m_i'=\phi(m_i)$ is a generator for $M_i'$. Let the numbers $\alpha_{i+1}^{1,2,3}$ be given by $m_ix_{1,2,3} = \alpha_{i+1}^{1,2,3} m_{i+1}$. Then $m_i'x_{1,2} = \phi(m_ix_{1,2}) = \alpha_{i+1}^{1,2} m_{i+1}'$. Let the numbers $\alpha_{i+1}^4$ be given by $m_i'x_4 = \alpha_{i+1}^4m_{i+1}'$. We define multiplication of $x_4$ on $m_i$ as $m_ix_4 = \alpha_{i+1}^4m_{i+1}$.

To see the above multiplication makes $M$ a $D$ module, we need to check that the relations of $D$ are compatible with $M$. Since $D$ is a quadratic algebra, we only have to verify all quadratic relations of $D$ are compatible with $M$. Take $f \in (I_D)_2$, then we have $f = f_A + f_B$, with $f_A \in (I_A)_2$ and $f_B \in (I_B)_2$. We check
\[ m_i f = m_i f_A + m_i f_B = m_i f_A + \phi^{-1}(m_i' f_B) = 0\]
since $m_if_A$ and $m_i'f_B$ both equal to 0. Thus $M$ is a $D$ module. It is easy to see that $\phi$ extend to an isomorphism of $M$ and $M'$ as $B$ modules. Since $\dim_k M_i =1$ and $M = M_0A \subset M_0D$, $M$ is a point module of $D$.
\end{proof}

By the above lemma, any $D$ point module that does not come from a compatible pair of $A$ and $B$ point modules is either a $(D,notA)$ module, or a $(D,notB)$ module. We have the following partial result on $(D,notA)$ and $(D,notB)$ modules.

\begin{lemma} \label{e3e4}
Take $D=A\cup_CB$. Let $e_3 = (0,0,1,0) \in \mathbb{P}^2$, and $e_4 = (0,0,0,1) \in \mathbb{P}^2$. Then $D$ has two point modules $E$ and $E'$, both $(D,notA)$ and $(D,notB)$ modules, given by the sequences $(e_3,e_4,e_3,e_4,\ldots)$ and $(e_4,e_3,e_4,e_3,\ldots)$.
\end{lemma}
\begin{proof}
We check the two sequences define $D$ modules. It suffices to show that the monomials $x_3^{0,1}(x_4x_3)x_4^{0,1}$ are not zero in $D$. Since $x_3x_4$ and $x_4x_3$ do not appear in the defining relations of $D$, by an argument similar to the proof of lemma \ref{growth}, the monomials $x_3^{0,1}(x_4x_3)x_4^{0,1}$ do not appear in any relations of $D$. So the sequences $(e_3,e_4,e_3,e_4,\ldots)$ and $(e_4,e_3,e_4,e_3,\ldots)$ define two $D$ modules $E$ and $E'$. It is then easy to see that $E$ and $E'$ are point modules of $D$, and are $(D,notA)$ and $(D,notB)$ modules.
\end{proof}

We note that by an extension of lemma \ref{AnotC}, if $M$ is a $(D,notA)$ module, then there is some $n$ such that $e_4$ represents the multiplication by $D_1$ of $M_n$. Unlike the case with dimension three regular algebras, it is not known in general if a point module of a dimension four regular algebra is uniquely determined by a single coordinate. To state our next result, we assume the correspondence is unique.

\begin{corollary}\label{pmD}
Let $D$ be a pushout algebra. Let $M$ and $M'$ be two point modules of $D$, with $M$ represented by the the sequence $(\alpha^1, \alpha^2, \ldots )$, and $M'$ represented by the sequence $(\beta^1, \beta^2, \ldots )$. If in addition we have $\alpha^i \ne \beta^i$ for all $i$, for any choice of non-isomorphic $M$ and $M'$. Then the modules represented by $(e_3,e_4,e_3,e_4,\ldots)$ and $(e_4,e_3,e_4,e_3,\ldots)$ are the only point modules of $D$ that does not come from a pair of compatible point modules of the algebras $A$ and $B$. This holds in particular for any noetherian pushout algebras.
\end{corollary}

%%%%%%%%%%%%%%%%%%%%
%%%%%%%%%%%%%%%%%%%%
\section{Type A Pushout Algebras} \label{sectiontypeA}

We observe that in the definition of pushout algebra \ref{defpushout}, the algebra 
\[ A = k \la x_1, x_2, x_3 \ra / (r_1, r_2, r_3) \]
is required to have three relations of the form
\begin{eqnarray*}
r_1 &=& x_3^2 + k_{32}^1x_3x_2 + k_{31}^1x_3x_1 + k_{23}^1x_2x_3 + k_{13}^1x_1x_3 + f_1 \\
r_2 &=& k_{32}^2x_3x_2 + k_{31}^2x_3x_1 + k_{23}^2x_2x_3 + k_{13}^2x_1x_3 + f_2 \\
r_3 &=& k_{32}^3x_3x_2 + k_{31}^3x_3x_1 + k_{23}^3x_2x_3 + k_{13}^3x_1x_3 + f_3
\end{eqnarray*}
Here the functions $f_1, f_2, f_3$ are quadratic functions in $x_1$ and $x_2$ only. We define a $2\times 2$ scalar matrix 
\[ K = \begin{bmatrix} k_{32}^2 & k_{31}^2 \\ k_{32}^3 & k_{31}^3 \end{bmatrix}. \]
We would like to state that the matrix $K$ must be non-singular if $A$ is a regular algebra. This is true for all the generic algebras listed in \cite{AS}. Since there is no complete list of the non-generic regular regular algebras, we can not rule out the possibility of the matrix $K$ being singular. For the purpose of studying pushouts, we assume that we are working with regular algebras with the matrix $K$ being non-singular. In this case, after take possible linear combination of $r_2$ and $r_3$, we can choose our relations for algebra $A$ to be 
\begin{eqnarray*}
r_1 &=& x_3^2 + k_{23}^1x_2x_3 + k_{13}^1x_1x_3 + f_1 \\
r_2 &=& x_3x_2 + k_{23}^2x_2x_3 + k_{13}^2x_1x_3 + f_2 \\
r_3 &=& x_3x_1 + k_{23}^3x_2x_3 + k_{13}^3x_1x_3 + f_3
\end{eqnarray*}
By the same argument, we choose the relations of algebra $B$ to be
\begin{eqnarray*}
r_4 &=& x_4^2 + k_{24}^4x_2x_4 + k_{14}^4x_1x_4 + f_4 \\
r_5 &=& x_4x_2 + k_{24}^5x_2x_4 + k_{14}^5x_1x_4 + f_5 \\
r_6 &=& x_4x_1 + k_{24}^6x_2x_4 + k_{14}^6x_1x_4 + f_6
\end{eqnarray*}

At this point we do not have a method to handle pushouts with $C$ a general dimension three cubic algebra. We next give a rough estimate of the Hilbert functions of our pushout algebra $D$. We introduce an ordering on the generators as $x_4>x_3>x_2>x_1$, and list our relations in descending order. We note that for the algebras $A$ and $B$, we have identified the leading term of each relation. We will now investigate the possible leading terms of the relations for algebra $C$. If the leading term of a relation of $C$ start with $x_1$, then all the rest of the terms must also start with $x_1$, and we can see $C$ is not a domain. So the leading term of any relation of $C$ must start with $x_2$. We list all possible leading terms of relations of $C$ as
\[  x_2^3, x_2^2x_1, x_2x_1x_2, x_2x_1^2  \]
Let $u_1$ and $u_2$ be two terms from the above list. Then we can represent the two relations of algebra $C$ as
\begin{eqnarray*}
r_7 &=& u_1 + g_1 \\
r_8 &=& u_2 + g_2
\end{eqnarray*}
with $g_1$ and $g_2$ cubic functions in $x_1$ and $x_2$ consist of terms with lower order monomials than $u_1$ and $u_2$, respectively. We define the algebra $D'$ as
\[ D' = k \la x_1, x_2, x_3, x_4 \ra / (x_3^2, x_3x_2, x_3x_1, x_4^2, x_4x_2, x_4x_1, u_1, u_2 )  \]
It is an easy observation that in each degree, we have $\dim_k D_n \le \dim_k D'_n  $. It is possible that $\dim_k D_n < \dim_k D'_n$. We note that if $D$ is regular, then it has Hilbert series $H_D(t) = (1-t)^{-4}$, so in particular $\dim_k D_4 = 35$ and $\dim_k D_5 = 56$. We have three different cases to consider here. 
\begin{enumerate}
\item If $u_1 = x_2x_1x_2$ and $u_2 = x_2x_1^2$, then $\dim_k D'_4 = 34$, so the algebra $D$ can not be regular.
\item If $u_1 = x_2^2x_1$ and $u_2 = x_2x_1^2$, then the algebra $D'$ has Hilbert series $(1-t)^{-4}$. A regular pushout algebra $D$ corresponding to this situation must then have all of its ambiguities resolvable in degrees higher than three.
\item For all other combinations of $u_1$ and $u_2$, either $\dim_k D'_4 > 35$, or $\dim_k D'_5 > 56$. In these cases any regular algebra $D$ must have non-trivial ambiguities in degrees higher than three. 
\end{enumerate}
We focus our study on the second possibility, namely $u_1 = x_2^2x_1$ and $u_2 = x_2x_1^2$, as this is the more promising case for having regular pushout algebras. We define a {\it type A pushout algebra} as follows.
\begin{definition} \label{type A}
Let $r_1, \ldots, r_6$ be quadratic relations and $r_7, r_8$ be cubic relations defined by
\begin{eqnarray*}
r_1 &=& x_3^2 - f_1(x_1,x_2,x_3) \\
r_2 &=& x_3x_1 - f_2(x_1,x_2,x_3) \\
r_3 &=& x_3x_2 - f_3(x_1,x_2,x_3) \\
r_4 &=& x_4x_1 - f_4(x_1,x_2,x_4) \\
r_5 &=& x_4x_2 - f_5(x_1,x_2,x_4) \\
r_6 &=& x_4^2 - f_6(x_1,x_2,x_4) \\
r_7 &=& x_2^2x_1 - f_7(x_1, x_2) \\
r_8 &=& x_2x_1^2 - f_8(x_1, x_2),
\end{eqnarray*}
where $f_1, \ldots, f_6$ consist of lower order quadratic monomials in the given variables, and $f_7, f_8$ consist of lower order cubic monomials in $x_1$ and $x_2$. Define the algebras $A, B$, and $C$ as 
\[ A = k \la x_3, x_1, x_2 \ra / (r_1, r_2, r_3) \]
\[ B = k \la x_1, x_2, x_4 \ra / (r_4, r_5, r_6) \]
\[ C = k \la x_1, x_2 \ra / (r_7, r_8) \]
If the algebras $A$ and $B$ are regular algebras, and $C$ is a regular subalgebra of $A$ and $B$. Then we call their pushout $D = A \cup_{C} B$ as a {\it type A pushout algebra}, or simply a {\it type A} algebra. 
\end{definition}
Here we list our relations in a slightly unusual order. This is done to makes some of our later computations easier. For the remainder of this paper, we work with type A algebras only. We fix the algebras $A, B, C,$ and $D$ as in the above definition. We now verify that a type A algebra has Hilbert series $(1-t)^{-4}$.

\begin{lemma} \label{lem 3}
Let $D$ be a type A algebra. Resolving ambiguities in $D$ does not generate additional relations other than $r_7$ and $r_8$. The algebra $D$ has Hilbert series $H_D(t) = (1-t)^{-4}$ and basis elements 
\[ x_1^p (x_2x_1)^l x_2^m  x_3^{0,1} (x_4x_3)^n x_4^{0,1}.\]
\end{lemma}
\begin{proof}
Since $A$ is a regular algebra, by \cite[lemma 1.2]{AS} it has Hilbert Series $H_{A}(t) = (1-t)^{-3}$. Define the algebra $A'$ to be
\[ A' = k \la x_3, x_1, x_2 \ra / (x_3^2, x_3x_1, x_3x_2, x_2^2x_1, x_2x_1^2). \]
Then it is easy to see that $H_{A'}(t) = (1-t)^{-3}$. By comparing the possible monomials of $A$ with those of $A'$ we conclude that resolving ambiguities in $A$ does not generate additional relations other than $r_7$ and $r_8$. A similar argument holds for the algebra $B$. It is easy to see that ambiguities of $D$ are the same as ambiguities of $A$ and $B$, hence the first part of our lemma. 

To see the second part of our lemma, we note that since $D$ has no extra ambiguities, the algebra $D$ has the same Hilbert Series and basis as the algebra
\[ k\la x_3, x_1, x_2, x_4\ra / (x_4^2, x_4x_2, x_4x_1, x_3^2, x_3x_2, x_3x_1, x_2^2x_1, x_2x_1^2 ). \]
This gives us the basis elements
\[ x_1^p (x_2x_1)^l x_2^m  x_3^{0,1} (x_4x_3)^n x_4^{0,1}\]
where $p, l, m, n$ are non-negative integers. From these basis elements computing the Hilbert Series is a straightforward combinatorics exercise.
\end{proof}

\begin{corollary}
If $D$ is a type A pushout algebra, then the algebras $A$ and $B$ are both subalgebras of $D$.
\end{corollary}

We define row vectors $X_A = [x_3, x_1, x_2]$, $X_B = [x_1, x_2, x_4]$, and $T_A, T_B$ to be $3\times 3$ matrices with entries in $A_1$ and $B_1$, with $T_AX_A^t = [r_1, r_2, r_3]^t$ and $T_BX_B^t = [r_4, r_5, r_6]^t$. Since $A$ and $B$ are regular, we have $k$ resolutions
\begin{equation} \label{**A}
0 \rightarrow A(-3) \buildrel{X_A^t} \over \longrightarrow  A(-2)^{\oplus 3} \buildrel{Q_AT_A} \over \longrightarrow  A(-1)^{\oplus 3}  \buildrel{X_A} \over \longrightarrow  A  \rightarrow  k_A  \rightarrow 0
\end{equation}
\begin{equation} \label{**B}
0 \rightarrow B(-3) \buildrel{X_B^t} \over \longrightarrow  B(-2)^{\oplus 3} \buildrel{Q_BT_B} \over \longrightarrow  B(-1)^{\oplus 3}  \buildrel{X_B} \over \longrightarrow  B  \rightarrow  k_B  \rightarrow 0
\end{equation}
Here $Q_A, Q_B$ are two $3\times 3$ non-singular scalar matrices.

For the following results we assume our pushout algebra $D$ has a potential resolution as given in definition \ref{PoRes},
\begin{equation} \label{**D}
0 \rightarrow D(-4) \buildrel{X^t} \over \longrightarrow  D(-3)^{\oplus 4} \buildrel{ST} \over \longrightarrow  D(-2)^{\oplus 6}  \buildrel{R} \over \longrightarrow   D(-1)^{\oplus 4}  \buildrel{X} \over \longrightarrow  D  \rightarrow  k_D  \rightarrow 0
\end{equation}
Here $X = [x_3,x_1,x_2,x_4]$, and $T$ satisfies $TX^t = [r_1, \ldots, r_6] ^t$. We observe that the upper left and lower right $3 \times 3$ blocks of the matrix $T$ is respectively the matrices $T_A$ and $T_B$, with the remaining entries 0. From the exact sequences (\ref{**A}) and (\ref{**B}) we have $\im (X_A^t) = \ker (T_A)$ and $\im (X_B^t) = \ker (T_B)$. We have the following lemma.

\begin{lemma} \label{lem 4}
Take $[d_1, d_2, d_3]^t \in D^{\oplus 3}$ with $T_A[d_1, d_2, d_3]^t = 0$. Then 
\[ [d_1, d_2, d_3]^t \in [x_3, x_1, x_2]^t D \]
\end{lemma}
\begin{proof}
In lemma \ref{lem 3}, we have a basis element of $D$ with the form
\[ x_1^p (x_2x_1)^l x_2^m  x_3^{0,1} (x_4x_3)^n x_4^{0,1}. \]
Define $g_{2n} = (x_4x_3)^n$ and $g_{2n+1} = (x_4x_3)^nx_4$. Then an element $d \in D$ can be uniquely written as
\[ d = \sum a_n g_n \]
with each $a_n \in A$ a sum of basis elements in $A$. Take another element $a' \in A$. Then
\[ a' d = a' (\sum a_n g_n) = \sum (a'a_n) g_n = \sum a'_n g_n. \]
Here $a'_n = a'a_n \in A$ is also a sum of basis elements in $A$, and $a'_n g_n$ is a sum of basis elements in $D$. Following this notation, we write
\[  [d_1, d_2, d_3]^t = [ \sum a_{1n}g_{n},  \sum a_{2n}g_{n}, \sum a_{3n}g_{n} ]^t \]
Let $[t_1, t_2, t_3] \in A^{\oplus 3}$ be a row in $T_A$. Then
\begin{eqnarray*}
0 &=& [t_1, t_2, t_3] [d_1, d_2, d_3]^t \\
&=& [t_1, t_2, t_3] [ \sum a_{1n}g_{n},  \sum a_{2n}g_{n}, \sum a_{3n}g_{n} ]^t \\
&=& \sum t_1a_{1n}g_{n} + \sum t_2a_{2n}g_{n} + \sum t_3a_{3n}g_{n} \\
&=& \sum (t_1a_{1n} + t_2a_{2n} + t_3a_{3n}) g_{n}.
\end{eqnarray*}
From the above we conclude that for each $n$, $0 = t_1a_{1n} + t_2a_{2n} + t_3a_{3n}$. Since $\im (X_A^t) = \ker (T_A)$, we have for each $n$ that
\[ [a_{1n}, a_{2n}, a_{3n}]^t \in [x_3, x_1, x_2]^t A. \]
This shows that
\[ [d_1, d_2, d_3]^t = \sum [a_{1n}, a_{2n}, a_{3n}]^t g_n \in  [x_3, x_1, x_2]^t D. \]
\end{proof}

Applying the same argument to $B \subset D$, we have
\begin{lemma} \label{lem 5}
Take $[d_1, d_2, d_3]^t \in D^{\oplus 3}$ with $T_B[d_1, d_2, d_3]^t = 0$. Then 
\[ [d_1, d_2, d_3]^t \in [x_1, x_2, x_4]^t D. \]
\end{lemma}

\begin{theorem} \label{thm 1}
Let $D$ be a type A algebra. If $D$ has a potential resolution, then $\gldim(D)=4$.
\end{theorem}

\begin{proof}
By lemma \ref{lem 3}, $D$ has Hilbert Series $H_D(t) = (1-t)^{-4}$. We can easily see that left multiplication by $x_1$ is injective. Hence left multiplication by $X^t$ is injective. By lemma \ref{lem 1}, to show that $\gldim(D) = 4$, we need only show that $\ker(T_D) \subset \im(X^t)$. 

Define matrices $T_1, T_2$ by
\[ T_1 = \begin{bmatrix} T_A & \begin{matrix} 0 \\ 0 \\ 0 \end{matrix} \\ \begin{matrix} 0 & 0 & 0 \\ 0 & 0 & 0 \\ 0 & 0 & 0 \end{matrix} & \begin{matrix} 0 \\ 0 \\ 0 \end{matrix} \end{bmatrix} , \ \ \ T_2 = \begin{bmatrix} \begin{matrix} 0 \\ 0 \\ 0 \end{matrix} & \begin{matrix} 0 & 0 & 0 \\ 0 & 0 & 0 \\ 0 & 0 & 0 \end{matrix} \\ \begin{matrix} 0 \\ 0 \\ 0 \end{matrix} & T_B \end{bmatrix} \]
with $T_A$ and $T_B$ as defined in (\ref{**A}) and (\ref{**B}). We have $T_D = T_1+T_2$, and by lemmas \ref{lem 4} and \ref{lem 5}
\[ \ker( T_1  ) = \begin{bmatrix} x_3 d_1 \\ x_1 d_1 \\ x_2 d_1 \\ d_2 \end{bmatrix} \]
\[ \ker( T_2  ) = \begin{bmatrix} d_4 \\ x_1d_3 \\ x_2d_3 \\ x_4d_3 \end{bmatrix}. \]
From the above we have that
\[ \ker (T_D) \subset \ker(T_1) \cap \ker(T_2) = \begin{bmatrix} x_3 d_1 \\ x_1 d_1 \\ x_2 d_1 \\ d_2 \end{bmatrix} \cap \begin{bmatrix} d_4 \\ x_1d_3 \\ x_2d_3 \\ x_4d_3 \end{bmatrix} = \im (\begin{bmatrix} x_3 \\ x_1 \\ x_2 \\ x_4 \end{bmatrix}. ) \]
\end{proof}

In lemma \ref{lem 7} we will show that all generic type A algebras have potential resolutions. In lemma \ref{lem 6} we will show that any type A algebra $D$ satisfies $D^{op} \cong D$. Thus combining theorem \ref{thm 1} and corollary \ref{cor 1} we show that a type A algebra is regular. Thus, up to the complete list of type A algebras, we have our main theorem.

\begin{theorem} \label{thm 2}
Type A pushout algebras are generically Artin-Schelter regular of global dimension four.
\end{theorem}

%%%%%%%%%%%%%%%%%%%%
%%%%%%%%%%%%%%%%%%%%
\section{Classification of Algebra $C$} \label{algebra C}

In this section we identify all possible dimension three regular cubic algebras $C$ with
\[ C = k\la x_1, x_2 \ra / (r_7 , r_8) \]
where the relations $r_7$ and $r_8$ are of the form
\begin{eqnarray*}
r_7 &=& x_2^2x_1 - c_1x_2x_1x_2 - c_2x_1x_2^2 - c_3x_1x_2x_1 - c_4x_1^2x_2 - c_5x_1^3  \\
r_8 &=& x_2x_1^2 - c_6x_1x_2^2 - c_7x_1x_2x_1 - c_8x_1^2x_2 - c_9x_1^3
\end{eqnarray*}
with coefficients $c_i \in k$.

Resolving the ambiguity $x_2^2x_1^2$ in $C$ gives us the equation 
\begin{eqnarray*}
0 &=& (x_2^2x_1)x_1 - x_2(x_2x_1^2) \\
&=& -c_6x_2x_1x_2^2+(c_1-c_7)x_2x_1x_2x_1-c_6c_8x_1x_2^3 \\
& & +(c_1c_2-c_1c_6c_9-c_7c_8)x_1x_2x_1x_2+(c_2^2+c_3c_6-c_8^2-c_6c_7c_9-c_2c_6c_9)x_1^2x_2^2 \\
& & +(c_2c_3+c_4+c_3c_7-c_3c_6c_9-c_8c_9-c_7^2c_9)x_1^2x_2x_1 \\
& & +(c_2c_4+c_3c_8-c_4c_6c_9-c_8c_9-c_7c_8c_9)x_1^3x_2 \\
& & +(c_5+c_2c_5+c_3c_9-c_5c_6c_9-c_9^2-c_7c_9^2)x_1^4.
\end{eqnarray*}
Solving for the coefficients $c_i$ in the above equation gives the following set of equations,
\begin{eqnarray*} \label{system C}
c_6 &=& 0 \\
c_7 &=& c_1 \\
0 &=& c_1(c_2-c_8) \\
0 &=& (c_2-c_8)(c_2+c_8) \\
0 &=& (c_1+c_2)c_3+c_4-(c_1^2+c_8)c_9 \\
0 &=& c_2c_4+c_3c_8-(1+c_1)c_8c_9 \\
0 &=& (c_2+1)c_5+c_3c_9-(1+c_1)c_9^2.
\end{eqnarray*}
Some of the solutions to the above system give algebras that are not a domain and hence not regular. We have the following lemma for the other solutions.

\begin{lemma} \label{lemC}
Any domain satisfying the above system of equations is regular.
\end{lemma}
\begin{proof}
To verify that a solution is a domain is computationally intensive. Instead we exclude the algebras that are obviously not domains. We then compute the Hilbert Series and construct a resolution $P^. \rightarrow k_C$ for each remaining algebra. We skip the computations here as they are straight forward and play no future role. For each solution algebra $C$, we check that $C^{op} \cong C$, and the dual resolution of $P^. \rightarrow k_C$ gives us (AS2).
\end{proof}

We list the relations of regular algebras $C$ below. It is worth noting that the algebras are only determined up to linear isomorphisms. We try to choose the basis that gives the shortest relations, but our choices are by no means optimal. We will see in the next section that only the algebras $C0, C1, C2$ can be subalgebras of a regular algebra $A$ as defined in \ref{type A}. This does not exclude the $Ci$'s below from being subalgebras of other quadratic regular algebras.

\begin{enumerate}

\item[C0]
This is a special case of the next algebra, C2.
\begin{eqnarray*}
r_7 &=& x_2^2x_1 - x_1x_2^2   \\
r_8 &=& x_2x_1^2 - x_1^2x_2 
\end{eqnarray*}

\item[C1]
\[ c_2 \ne 0 \]
\begin{eqnarray*}
r_7 &=& x_2^2x_1 - c_1 x_2x_1x_2 - c_2 x_1x_2^2   \\
r_8 &=& x_2x_1^2 - c_1 x_1x_2x_1 - c_2 x_1^2x_2 
\end{eqnarray*}

% c6=0, c7=c1, c1=c1, c2=c2, c8=c2, c3=0, c4=0, c9=0, c5=0,

\item[C2]
\[ c_3 \ne 0 \]
\begin{eqnarray*}
r_7 &=& x_2^2x_1 - 2x_2x_1x_2 + x_1x_2^2 - c_3 x_1x_2x_1 + c_3 x_1^2x_2 - c_5 x_1^3  \\
r_8 &=& x_2x_1^2 - 2x_1x_2x_1 + x_1^2x_2
\end{eqnarray*}

% c6=0, c7=c1, c1=2, c2=-1, c8=-1, c9=0, c3<>0, c4=-c3, 

\item[C3]
\[ c_2 \ne 0 \]
\begin{eqnarray*}
r_7 &=& x_2^2x_1 - c_2 x_1x_2^2  \\
r_8 &=& x_2x_1^2 + c_2 x_1^2x_2
\end{eqnarray*}

% c6=0, c7=c1, c1=0, c2<>0, c8=-c2, c3=0, c4=0, c5=0, c9=0,

\item[C4]
\begin{eqnarray*}
r_7 &=& x_2^2x_1 + x_1x_2^2 - x_1^3  \\
r_8 &=& x_2x_1^2 - x_1^2x_2 
\end{eqnarray*}

% c6=0, c7=c1, c1=0, c2=-1, c8=1, c3=0, c4=0, c5<>0, c9=0, can pick c5=1.

\item[C5]
\[ c_1 \ne 2 \]
\begin{eqnarray*}
r_7 &=& x_2^2x_1 - c_1 x_2x_1x_2 + x_1x_2^2 - c_5 x_1^3  \\
r_8 &=& x_2x_1^2 - c_1 x_1x_2x_1 + x_1^2x_2
\end{eqnarray*}

% if c1=2, this is a special case of C6.
% c6=0, c7=c1, c1=c1, c2=-1, c8=-1, c3=0, c4=0, c5<>0, c9=0,

\item[C6]
\[ c_9 \ne 0 \]
\begin{eqnarray*}
r_7 &=& x_2^2x_1 - 2x_2x_1x_2 + x_1x_2^2 - 3c_9 x_1x_2x_1  \\
r_8 &=& x_2x_1^2 - 2x_1x_2x_1 + x_1^2x_2 - c_9 x_1^3
\end{eqnarray*}

% c6=0, c7=c1, c1=2, c2=-1, c8=-1, c3=3c_9 , c4=0, c5=0, c9<>0,

\item[C7] 
One of $c_4, c_9$ is non-zero.
\begin{eqnarray*}
r_7 &=& x_2^2x_1 - x_1x_2^2 - (c_9-c_4)x_1x_2x_1 - c_4 x_1^2x_2 - \frac{1}{2}c_4c_9 x_1^3  \\
r_8 &=& x_2x_1^2 - x_1^2x_2 - c_9 x_1^3
\end{eqnarray*}

% c6=0, c7=c1, c1=0, c2=1, c8=1, c3=c9-c4, c5=c4*c9/2, one of c4, c9 non-zero.

\end{enumerate}

%%%%%%%%%%%%%%%%%%%%%%%%%%%%%%%

\section{Classification of Algebra $A$} \label{algebra A}

In this section we identify all possible regular quadratic algebras $A$ over each choice of algebra $C$. By symmetry the algebras $B$ have the same classification. We list the relations $r_i$ of $A$, which we recall was defined as $A = k \la x_3, x_1, x_2 \ra / ( r_1, r_2, r_3)$.
\begin{eqnarray*}
r_1 &=& x_3^2  -  a_{6}x_2x_3 -  a_{5}x_2^2 -  a_{4}x_2x_1 -  a_{3}x_1x_3 -  a_{2}x_1x_2 -  a_{1}x_1^2 \\
r_2 &=& x_3x_1 - a_{26}x_2x_3 - a_{25}x_2^2 - a_{24}x_2x_1 - a_{23}x_1x_3 - a_{22}x_1x_2 - a_{21}x_1^2 \\
r_3 &=& x_3x_2 - a_{16}x_2x_3 - a_{15}x_2^2 - a_{14}x_2x_1 - a_{13}x_1x_3 - a_{12}x_1x_2 - a_{11}x_1^2 
\end{eqnarray*}
We reduce the relations by a linear change of variables. Any change of variables must preserve the relations of $C$ inside $A$. Thus we can scale $x_2$ or change $x_3$ to $x_3 + \alpha x_1 + \beta x_2$. If $a_{26}=0$ then we can choose $a_{24}=a_{12}=0$ after the change of variable $x_3 \rightarrow x_3 + (a_{13}a_{24}+a_{12})x_1+a_{24}x_2$. Otherwise $a_{26} \ne 0$ and first we can choose $a_{26}=1$ by scaling $x_2$, then we can choose $a_{24}=a_{25}=0$ after the change of variable $x_3 \rightarrow (x_3-(a_{24}+a_{25})x_1-a_{25}x_2)$. We have
\begin{eqnarray*}
r_1 &=& x_3^2  -  a_{6}x_2x_3 -  a_{5}x_2^2 -  a_{4}x_2x_1 -  a_{3}x_1x_3 -  a_{2}x_1x_2 -  a_{1}x_1^2 \\
r_2 &=& x_3x_1 - a_{26}x_2x_3 - a_{25}x_2^2                - a_{23}x_1x_3 - a_{22}x_1x_2 - a_{21}x_1^2 \\
r_3 &=& x_3x_2 - a_{16}x_2x_3 - a_{15}x_2^2 - a_{14}x_2x_1 - a_{13}x_1x_3 - a_{12}x_1x_2 - a_{11}x_1^2
\end{eqnarray*}
with either $a_{26}=a_{12}=0$, or $a_{25}=0$ and $a_{26}=1$.

In the special case of an algebra $A$ over $C0$, in addition to the previous change of variables we can also change $x_2$ to $x_2 + \gamma x_1$. In this case we can reduce $r_1, r_2, r_3$ further to get
\begin{eqnarray*}
r_1 &=& x_3^2  -  a_{6}x_2x_3 -  a_{5}x_2^2 -  a_{4}x_2x_1 -  a_{3}x_1x_3 -  a_{2}x_1x_2 -  a_{1}x_1^2 \\
r_2 &=& x_3x_1 - a_{26}x_2x_3 - a_{25}x_2^2                - a_{23}x_1x_3 - a_{22}x_1x_2 - a_{21}x_1^2 \\
r_3 &=& x_3x_2 - a_{16}x_2x_3 - a_{15}x_2^2 - a_{14}x_2x_1 - a_{13}x_1x_3 - a_{12}x_1x_2 - a_{11}x_1^2 \\
r_7 &=& x_2^2x_1 - x_1x_2^2  \\
r_8 &=& x_2x_1^2 - x_1^2x_2,
\end{eqnarray*}
with either $a_{26}=1$ and $ a_{25}=a_{23}=0$, or $ a_{13}=a_{26}=a_{12}=0$.

\begin{remark} We use $Aij$ to denote the algebra of type $j$, lying over $Ci$. \end{remark}

Only the algebras $C0$, $C1$, and $C2$ have regular algebra $A$ containing them. We have the following lemma.

\begin{lemma} \label{lem 6}
Let the algebras $C$ and $A$ be defined as in \ref{type A}. Then the regular algebras $A21$, $A11$, $A12$, $A13$, $A01$, $A02$, $A03$, $A04$ listed below are a complete classification of the algebra $A$. In addition, for each algebra $A$ we have $A^{op} \cong A$ by the change of variable $x_1$ to $-x_1$. Thus for any type A algebra $D$, we also have that $D^{op} \cong D$ by the same change of variable. Any point module of $A$ is an $(A,C)$ point module.
\end{lemma} 
\begin{proof}
We give an outline of our method of computation. Similar to the computation for lemma \ref{lemC}, we resolve all ambiguities. By lemma \ref{lem 3}, the ambiguities we have to check are $x_3^3, x_3^2x_1, x_3^2x_2, x_3x_2^2x_1, x_3x_2x_1^2,$ and $x_2^2x_1^2$. The first three ambiguities come from the quadratic relations of $A$, and the last three ambiguities come from the extra cubic relations of $C \subset A$. At the end of this step we have a large system of equations in the coefficients $a_i$, which we then solve.

For each solution we check weather or not it is a domain. Then we check if the quadratic relations generate the cubic relations $r_7$ and $r_8$. The last step is to find a linear resolution for $k_A$ for each algebra. This is equivalent to finding a non-singular $3\times 3$ scalar matrix $S_A$ with $R_AS_AX_A^t = 0$. We also observe that for each solution we have $A^{op} \cong A$. 

Finally we note none of the algebras in our list have a point module corresponding to $(0,0,1)$. So by lemma \ref{AnotC}, all point modules of algebras $A$ are $(A,C)$ point modules.
\end{proof}

\begin{remark} 
At this stage we have no clear indication of the relations between the resolutions $k_A$ and any potential resolution of $k_D$.  It is of interest to find out if forming a pushout algebra introduces additional constraints on the coefficients. In particular, see the remark following algebra $A04$.
\end{remark}

We list the possible regular algebras $A$ below, together with the relations $r_7$ and $r_8$ for the cubic subalgebra $C$ in $A$, and the point scheme of both algebras $A$ and $C$. By theorem \ref{pmdim3}, the point modules of the algebra $A$ can be represented by points $(\alpha_1, \alpha_2, \alpha_3) \in \mathbb{P}^2$. We list our point scheme as a cubic equation in $\alpha_1, \alpha_2, \alpha_3$. Similarly the point modules of the algebra $C$ can be represented by points $((\alpha_1, \alpha_2), (\beta_1, \beta_2)) \in \mathbb{P}^1\times \mathbb{P}^1$. We list our point scheme as a degree four equation in $\alpha_1, \alpha_2, \beta_1, \beta_2$.

We note the algebras in our list may have nicer representations as just quadratic regular algebras. To give an example, our algebra $A04$ from the list below is a regular Clifford algebra, hence by \cite[Corollary 4.8]{StV} can be represented with the relations
\[ xy+yx-\lambda_1z^2, xz+zx-\lambda_2z^2, yz+zy-\lambda_3x^2. \]
If we use these shorter relations instead, then we would have to find all the possible cubic subalgebras.

%% For 2 point modules to be compatible, they need to agree on the points ((\alpha_1, \alpha_2), (\beta_1, \beta_2) ). To check this is same as finding the automorphism.

\begin{enumerate}

\item[A21]
% p^2=1,a6=0,a5=0,a4=1,a3=0,a2=-1,a1=1-k,a26=0,a25=0,a23=p,a22=0,a21=0,a16=p,a15=0,a14=0,a13=-k*p,a12=0,a11=0,
\[ p^2 = 1 \]
\begin{eqnarray*}
r_1 &=& x_3^2 -  x_2x_1 + x_1x_2 + a_1 x_1^2 \\
r_2 &=& x_3x_1 - px_1x_3 \\
r_3 &=& x_3x_2 - px_2x_3 + p(1-a_1) x_1x_3 \\
r_7 &=& x_2^2x_1 - 2x_2x_1x_2 + x_1x_2^2 - 2 x_1x_2x_1 + 2 x_1^2x_2 + 2a_1(1-a_1) x_1^3  \\
r_8 &=& x_2x_1^2 - 2x_1x_2x_1 + x_1^2x_2
\end{eqnarray*}
% \[ \alpha_3(\alpha_3^2 +p(2a_1-1)\alpha_1^2) = 0 \]
% \[ (\alpha_1\beta_2-\alpha_2\beta1+a_1\alpha_1\beta_1)(\alpha_1\beta_2-\alpha_2\beta_1+(1-a_1)\alpha_1\beta_1) = 0 \]
%  x_3 normal. p*x3*(x3^2*p+2*x1^2*a1-x1^2)
\begin{comment}
\[ \beta_3(\beta_3^2 + p(2a_1-1)\beta_1^2) = 0  \]
\[ \alpha_3\beta_3 - \alpha_2\beta_1 + \alpha_1\beta_2 + a_1\alpha_1\beta_1 = 0\]
\[ \alpha_3\beta_1 - p \alpha_1\beta_3 = 0 \]
\[ \alpha_3\beta_2 - p \alpha_2\beta_3 + p(1-a_1)\alpha_1\beta_3 = 0 \]
\end{comment}

\begin{comment}
\[ S = \begin{bmatrix}
1 & 0 & 0   \\
0 & 1-k & 1   \\
0 & -1 & 0  
\end{bmatrix} \]
\end{comment}

\item[A11]
% a6=0,a5=0,a4=1,a3=0,a2=c2/(a23*a23),a1=0,a26=0,a25=0,a22=0,a21=0,a16=1/a23,a15=0,a14=0,a13=0,a12=0,a11=0,
\[a_{23} \ne 0, \ c_2 \ne 0 \]
\begin{eqnarray*}
r_1 &=& x_3^2  -  x_2x_1 - c_2a_{23}^{-2} x_1x_2 \\
r_2 &=& x_3x_1 - a_{23}x_1x_3 \\
r_3 &=& x_3x_2 - a_{23}^{-1} x_2x_3 \\
r_7 &=& x_2^2x_1 - (a_{23}^2-c_2a_{23}^{-2}) x_2x_1x_2 - c_{2}x_1x_2^2  \\
r_8 &=& x_2x_1^2 - (a_{23}^2-c_2a_{23}^{-2}) x_1x_2x_1 - c_{2}x_1^2x_2
\end{eqnarray*}
% \[ \alpha_3(\alpha_3^2 - (a_{23}+c_2a_{23}^{-3})\alpha_2\alpha_1) = 0  \]
% \[ (a_{23}^2-c_2a_{23}^{-2})(\alpha_2\beta_1 + c_2a_{23}^{-2}\alpha_1\beta_2)(\alpha_2\beta_1 - a_{23}^2\alpha_1\beta_2) = 0 \]
% x_3 normal. -x3*(c2*x2*x1+x1*a23^4*x2-x3^2*a23^3)/a23^3

\begin{comment}
\[ S = \begin{bmatrix}
1 & 0 & o   \\
0 & 0 & -c_2a_{23}^{-2}   \\
0 & -1 & 0  
\end{bmatrix} \]
\end{comment}

\item[A12]
% p^2=-1,a6=0,a5=1,a4=0,a3=0,a2=0,a1=0,a26=0,a25=0,a23=p,a22=0,a21=0,a16=1,a15=0,a14=0,a13=0,a12=0,a11=1,
\[ p^2 = -1 \]
\begin{eqnarray*}
r_1 &=& x_3^2  -  x_2^2   \\
r_2 &=& x_3x_1 - p x_1x_3  \\
r_3 &=& x_3x_2 - x_2x_3 - x_1^2 \\
r_7 &=& x_2^2x_1 + x_1x_2^2  \\
r_8 &=& x_2x_1^2 + x_1^2x_2
\end{eqnarray*}
% \[ \alpha_3^3 - \alpha_3\alpha_2^2 - p\alpha_2\alpha_1^2 = 0   \]
% \[ \mathbb{P}^1 \times \mathbb{P}^1  \]
% Potential skew Clifford. x_3 not normal. x2*x1^2-p*x3*x2^2+x3^3*p

\begin{comment}
\[ S = \begin{bmatrix}
0 & 0 & 1   \\
0 & -1 & 0   \\
-1 & 0 & 0  
\end{bmatrix} \]
\end{comment}

\item[A13]
% p^2-p+1=0,a6=0,a5=0,a4=1,a3=0,a2=-p,a1=0,a26=0,a23=1-p,a22=0,a21=0,a16=p,a15=0,a14=0,a13=0,a12=0,a11=0,
\[ p^2-p+1 =0,  \  a_{11}a_{25} \ne 1-p \]
\begin{eqnarray*}
r_1 &=& x_3^2 -  x_2x_1 + p x_1x_2  \\
r_2 &=& x_3x_1 - a_{25}x_2^2 - (1-p) x_1x_3 \\
r_3 &=& x_3x_2 - p x_2x_3 - a_{11}x_1^2 \\
r_7 &=& x_2^2x_1 - (p-1)x_1x_2^2  \\
r_8 &=& x_2x_1^2 - (p-1)x_1^2x_2
\end{eqnarray*}
% \[ \alpha_3^3 + (2p-2-a_{11}a_{25})\alpha_3\alpha_2\alpha_1 -pa_{25}\alpha_2^3 + a_{11}\alpha_1^3 = 0  \]
% \[ \mathbb{P}^1 \times \mathbb{P}^1  \]
% Potential Skew Clifford. -p*x2^3*a25+p^2*x2*x3*x1+x1^3*a11-x1*x3*x2+x1*x3*x2*p+x3^3*p-x3^3*p^2-x3*a25*x2*a11*x1

\begin{comment}
\[ S = \begin{bmatrix}
1 & 0 & 0   \\
0 & 0 & -p   \\
0 & 1 & 0
\end{bmatrix} \]
\end{comment}

\item[A01]

This is the special case of A11 with $a_{23}^4=1$ and $c_2 = 1$. If $p=-1$ this is also a special case of the algebra $A04$ later in this list.
\[ p^4=1 \]
\begin{eqnarray*}
r_1 &=& x_3^2  -  x_2x_1 -  p^2 x_1x_2 \\
r_2 &=& x_3x_1 - p x_1x_3 \\
r_3 &=& x_3x_2 - p^{-1} x_2x_3 \\
r_7 &=& x_2^2x_1 - x_1x_2^2  \\
r_8 &=& x_2x_1^2 - x_1^2x_2
\end{eqnarray*}
% \[ \alpha_3(\alpha_3^2 - 2p\alpha_2\alpha_1) = 0   \]
% \[ \mathbb{P}^1 \times \mathbb{P}^1  \]
% Potential skew Clifford. x3*(x3^2-2*p*x2*x1)

\item[A02]
% a6=a4=a3=a2=a1=0, a5=1, a26=a25=a24=a22=a21=0, a23=1, a16=-1, a14=a13=a12=0, a11=1,
\begin{eqnarray*}
r_1 &=& x_3^2 -  x_2^2 \\
r_2 &=& x_3x_1 - x_1x_3 \\
r_3 &=& x_3x_2 + x_2x_3 - a_{15}x_2^2 - x_1^2  \\
r_7 &=& x_2^2x_1 - x_1x_2^2  \\
r_8 &=& x_2x_1^2 - x_1^2x_2
\end{eqnarray*}
% \[ \alpha_3^3 - a_{15}\alpha_3^2\alpha_2 + \alpha_3\alpha_2^2 - \alpha_2\alpha_1^2 = 0 \]
% \[ \mathbb{P}^1 \times \mathbb{P}^1  \]
% Potential skew Clifford. x2*x1^2-x3*x2^2+x3^2*a15*x2-x3^3

\item[A03]
% a6=a4=a3=a2=0, a5=1, a1=1, a26=a25=a24=a22=a21=0, a23=1, a16=-1, a14=a13=a12=0,
\[ a_{11}(a_{15}-a_{11}) \ne 1 \]
\begin{eqnarray*}
r_1 &=& x_3^2  - x_2^2 -  x_1^2 \\
r_2 &=& x_3x_1 - x_1x_3 \\
r_3 &=& x_3x_2 + x_2x_3 - a_{15}x_2^2 - a_{11}x_1^2 \\
r_7 &=& x_2^2x_1 - x_1x_2^2  \\
r_8 &=& x_2x_1^2 - x_1^2x_2
\end{eqnarray*}
% \[ \alpha_3^3 - a_{15}\alpha_3^2\alpha_2 + \alpha_3\alpha_2^2 - \alpha_3\alpha_1^2 +(a_{15}-a_{11}) \alpha_2\alpha_1^2 = 0   \]
% \[ \mathbb{P}^1 \times \mathbb{P}^1  \]
% Potential skew Clifford. -x1^2*a15*x2+x1^2*x3+x2*x1^2*a11-x3*x2^2+x3^2*a15*x2-x3^3

\item[A04]
Generically regular, with the coefficients $a_i$ satisfying lemma \ref{generic},
\begin{eqnarray*}
r_1 &=& x_3^2  -  a_{5}x_2^2 -  a_{4}x_2x_1 -  a_{4}x_1x_2 -  a_{1}x_1^2 \\
r_2 &=& x_3x_1 - a_{25}x_2^2 + x_1x_3 - a_{21}x_1^2 \\
r_3 &=& x_3x_2 + x_2x_3 - a_{15}x_2^2 - a_{11}x_1^2 \\
r_7 &=& x_2^2x_1 - x_1x_2^2  \\
r_8 &=& x_2x_1^2 - x_1^2x_2
\end{eqnarray*}
% \[ \alpha_3^3 - a_{15} \alpha_3^2\alpha_2 + a_5\alpha_3\alpha_2^2 - a_{21}\alpha_3^2\alpha_1 + (2a_4+a_{15}a_{21}-a_{11}a_{25})\alpha_3\alpha_2\alpha_1 + a_1\alpha_3\alpha_1^2 + a_4a_{25}\alpha_2^3 + (a_1a_{25}-a_4a_{15}-a_5a_{21})\alpha_2^2\alpha_1 + (a_5a_{11}-a_1a_{15}-a_4a_{21})\alpha_2\alpha_1^2 + a_4a_{11}\alpha_1^3   \]
% \[ \mathbb{P}^1 \times \mathbb{P}^1  \]

\begin{remark}
The generic conditions on the coefficients for this algebra are precisely the conditions we need to guarantee that $A$ is a domain and that $C \subset A$. These generic conditions do cause us some difficulties in solving for a $k$-resolution for any type A pushout containing $A04$. We include the conditions below for the purpose of completion.  
\end{remark}

% This algebra is a Clifford Algebra. a4*x2^3*a25+2*a4*x2*x1*x3-a4*x2^2*x1*a15+a1*x1*a25*x2^2+a1*x1^2*x3-a1*x1^2*a15*x2+a5*x2*x1^2*a11+a5*x2^2*x3-a5*x2^2*a21*x1+a4*x1^3*a11-a4*x1^2*x2*a21+x3^3-x3^2*a15*x2-a21*x1*x3^2+x3*a21*x1*a15*x2-x3*a25*x2*a11*x1

\end{enumerate}

For the algebra $A04$, if we define
\begin{eqnarray*}
k_1 &=& -a_4a_{15}-a_1a_{25} \\ 
k_2 &=& a_4a_{21}+a_5a_{11} \\ 
k_3 &=& a_5 - a_{21}a_{25} \\
k_4 &=& a_4 + a_{11}a_{25} \\
k_5 &=& a_{11}a_{15} - a_1,
\end{eqnarray*}
then resolving the ambiguities $x_3^3$, $x_3^2x_1$, $x_3^2x_2$ give us the following cubic relations:
\begin{eqnarray*}
k_1 r_7 + k_2 r_8 &=& 0, \\
k_3 r_7 + k_4 r_8 &=& 0, \\
-k_4 r_7 + k_5 r_8 &=& 0.
\end{eqnarray*}
Thus for $r_1, r_2, r_3$ to generate $r_7, r_8$, we need the matrix
\[ \begin{bmatrix} k_1 & k_2 \\ k_3 & k_4 \\ -k_4 & k_5 \end{bmatrix} \]
to have rank 2.  So one of $k_1k_4-k_2k_3$, $k_1k_5+k_2k_4$, and $k_3k_5+k_4^2$ must be non-zero.

In addition to having $r_1, r_2, r_3$ generate $r_7$ and $r_8$, we need to solve for a $3\times 3$ non-singular scalar matrix $S_A$ with $R_AS_AX_A^t = 0$. Our solution has the form
\[ S_A = \begin{bmatrix}
s_1 & s_2 & s_3 \\
s_2 & -(a_{11}s_3+a_{21}s_2+a_1s_1) & -a_4s_1 \\
s_3 & -a_4s_1 & -(a_{15}s_3+a_{25}s_2+a_5s_1)
\end{bmatrix}\]
subject to the constraints
\begin{eqnarray*}
0 &=& -k_4 s_3 + k_3 s_2 + (- a_{15}k_4 + a_{25} k_5) s_1  \\
0 &=& k_5 s_3 + k_4 s_2 + (a_{11} k_3 + a_{21} k_4) s_1.  
\end{eqnarray*}
A solution for a non-singular $S_A$ is the same as a non-trivial solution for $s_1, s_2, s_3$, subjected to the above two equations with $\det(S_A) \ne 0$, and having one of $k_1k_4-k_2k_3$, $k_1k_5+k_2k_4$, or $k_3k_5+k_4^2$ be non-zero. This is true for generic choices of the coefficients $a_i$. We sum up the above argument in the following lemma,

\begin{lemma} \label{generic}
The algebra $A04$ is regular if there is a non-trivial set of solutions $s_1, s_2, s_3$ satisfying the following set of conditions,
\begin{eqnarray*}
0 &=& -k_4 s_3 + k_3 s_2 + (- a_{15}k_4 + a_{25} k_5) s_1  \\
0 &=& k_5 s_3 + k_4 s_2 + (a_{11} k_3 + a_{21} k_4) s_1 \\
0 &\ne & \det( S_A) \\
0 &\ne & k_1k_4-k_2k_3 \ \  \textrm{or} \ \  0 \ne k_1k_5+k_2k_4  \ \ \textrm{or} \ \  0 \ne k_3k_5+k_4^2,
\end{eqnarray*}
where $k_1= -a_4a_{15}-a_1a_{25}$, $k_2 = a_4a_{21}+a_5a_{11}$, $k_3 = a_5 - a_{21}a_{25}$, $k_4 = a_4 + a_{11}a_{25}$, $k_5= a_{11}a_{15} - a_1$, and
\[ S_A = \begin{bmatrix}
s_1 & s_2 & s_3 \\
s_2 & -(a_{11}s_3+a_{21}s_2+a_1s_1) & -a_4s_1 \\
s_3 & -a_4s_1 & -(a_{15}s_3+a_{25}s_2+a_5s_1)
\end{bmatrix}.\]
In particular the algebra $A04$ is regular for generic choices of coefficients $a_i$.
\end{lemma}
\begin{proof}
There are three variables $s_1, s_2, s_3$ with two equations
\begin{eqnarray*}
0 &=& -k_4 s_3 + k_3 s_2 + (- a_{15}k_4 + a_{25} k_5) s_1  \\
0 &=& k_5 s_3 + k_4 s_2 + (a_{11} k_3 + a_{21} k_4) s_1
\end{eqnarray*}
and two open conditions
\begin{eqnarray*}
0 &\ne & \det( S_A) \\
0 &\ne & k_1k_4-k_2k_3 \ \ \textrm{or}  0 \ne k_1k_5+k_2k_4 \ \ \textrm{or} \ \ 0 \ne k_3k_5+k_4^2.
\end{eqnarray*}
Hence there exist non-trivial solutions to $s_1, s_2, s_3$ for an open dense set of coefficients $a_i$ and the algebra $A04$ is regular for generic choice of $a_i$.
\end{proof}

%  This is \det(S_A) \[ a_{11}s_3^3+a_{21}s_2s_3^2+a_{15}s_2^2s_3+a_{25}s_2^3+(a_1+a_{11}a_{15})s_1s_3^2+(-2a_4+a_{11}a_{25}+a_{15}a_{21})s_1s_2s_3+(a_5+a_{21}a_{25})s_1s_2^2+(a_{11}a_{5}+a_1a_{15})s_1^2s_3+(a_{21}a_{5}+a_{1}a_{25})s_1^2s_2+(a_1a_5-a_4^2)s_1^3 \ne 0 \]

% \[ s_1(a_{11}s_3+a_{21}s_2+a_1s_1)(a_{15}s_3+a_{25}s_2+a_5s_1)-2a_4s_1s_2s_3-a_4^2s_1^3+s_2^2(a_{15}s_3+a_{25}s_2+a_5s_1)+s_3^2(a_{11}s_3+a_{21}s_2+a_1s_1) \ne 0\]

%%%%%%%%%%%%%%%%%%%%%%%%%%%

%%%%%%%%%%%%%%%%%%%%%%%%%%%%%%%%%%%%%%%%%%%%%

\section{List and Properties of type A Algebras} \label{algebra D}

By the previous section, we have a complete list of all subalgebras $A$. Hence we can write down the generating relations of each possible pushout $D$. We check explicitly that the generic algebras have potential resolutions. That is, for each generic algebra $D$, we find a non-singular $6 \times 6$ scalar matrix $S = [s_{ij}], \ i,j \in \{1\ldots6\}, s_{ij} \in k $ such that $RST = 0$. The following lemma will greatly aid us in our effort to solve for the matrix $S$. We refer back to \ref{PoRes} for the definitions of matrices $R, S, T$.

\begin{lemma}
If $D$ is a possible type A algebra, then its $6 \times 6$ matrix $S$ from \ref{PoRes} has the block form
\[ S = \begin{bmatrix}
S_1 & 0 \\ 0 & S_2 \end{bmatrix} \]
where $S_1$ and $S_2$ are $3\times 3$ non-singular matrices.
\end{lemma}

\begin{proof}
For this lemma we only need to use the leading term of each relation. Write
\[ D = k\la x_3, x_1, x_2, x_4 \ra /(x_3^2-f_1, x_3x_1-f_2, x_3x_2-f_3, x_4x_1-f_4, x_4x_2-f_5, x_4^2-f_6) \]
We expand the entry $[RST]_{12} = \sum R_{1i}s_{ij}T_{j2}$ in lowest Gr\"{o}bner basis. The only $x_3x_4$ term in it is $s_{14}x_3x_4$. Since $RST=0$ and there are no relations with $x_3x_4$, we must have $s_{14}=0$. Again in $[RST]_{12}$ we obtain an $x_1x_4$ term only if $i=2, j=4$. The monomials whose reductions give possible $x_1x_4$ terms are $x_4x_1, x_4x_2, x_4^2$. These do not appear in $[RST]_{12}$. Thus the $x_1x_4$ term in $[RST]_{12}$ is $s_{24}x_1x_4$, so we conclude that $s_{24}=0$. Similarly, the $x_2x_4$ term in $[RST]_{12}$ is $s_{34}x_2x_4$, so we conclude that $s_{34}=0$.

If we carry out the same analysis for the $x_3x_4$, $x_1x_4$, $x_2x_4$ terms in $[RST]_{13}$ and $[RST]_{14}$ we get 
\[ s_{15}=s_{25}=s_{35}=s_{16}=s_{26}=s_{36}=0. \] 
Thus the upper right $3\times 3$ block of $S$ consists of all zeros. Similarly, by computing the $x_1x_3, x_2x_3, x_4x_3$ terms in $[RST]_{41}, [RST]_{42}, [RST]_{43}$ we have the lower left $3\times 3$ block of $S$ to consist of all zeros. \\
\end{proof}

With the help of the preceding lemma, computing the matrix $S$ is straight forward. We state the following lemma omitting the computation details.
\begin{lemma} \label{lem 7}
A generic type A algebra have potential resolutions.
\end{lemma}
This concludes the proof of our main theorem \ref{thm 2}. Now having established the regularity of type A algebras, we move to study their other properties. 

\begin{remark}

\begin{enumerate}

\item
The matrices $S_1, S_2$ do not have to be the same as the $3\times 3$ matrices $S_A, S_B$ from the algebras $A$ and $B$. We notice from our computation that if we choose $\det(S_1) = 1$, then we always have $\det(S_2) = -1$. We do not yet know the reason behind this. 

\item
To find a potential resolution, we only need the existence of a solution matrix $S$. Presumably the large system of coefficients generated from $RST=0$ may be reduced to a small set of characterizing equations. This may give a classification of type A algebras. 
\end{enumerate}
\end{remark}

%%%%%%%%%%%%%%%%%%%%%%%%%%%%
\subsection{Symmetry and Inherited Reflection Automorphisms.}
We note it is clear from the general construction of a pushout algebra, that there is a symmetry between the subalgebras $A$ and $B$. So if we exchange the variables $x_3$ and $x_4$ in the algebra $D$, we obtain an isomorphic algebra. 

We further note that the algebras $A21$, $A11$, $A01$ all have a reflection automorphism $\sigma$, given by $\sigma(x_1)=x_1$, $\sigma(x_2)=x_2$, and $\sigma(x_3)= -x3$. This automorphism $\sigma$ can be naturally extended to the pushout $D$. To be more precise, for a pushout algebra $D$, we define linear maps $\sigma_i$, $i=1,\ldots,4,$ by $\sigma_i(x_i)= -x_i$, and $\sigma_i(x_j) = x_j$ for $j \ne i$. If $D$ has one of the algebras $A21$, $A11$, $A01$ as a subalgebra, then it is clear that $\sigma_3$ is an reflection automorphism of $D$, which can be viewed as an inherited reflection from the algebra $A$. By symmetry, if $D$ has one of the algebras $B21$, $B11$, $B01$ as a subalgebra, then $\sigma_4$ is an reflection automorphism of $D$. 

The algebras $A12$, $A02$, $A03$ also have a reflection automorphism $\sigma'$, given by $\sigma'(x_1)=-x_1$, $\sigma'(x_2)=x_2$, and $\sigma'(x_3)=x_3$. In this case the pushout $D$ has $\sigma_1$ as an automorphism only if it contains one of $A12$, $A02$, $A03$, and one of $B12$, $B02$, $B03$. 

In our list of type A algebras at the end, the algebras with the reflection $\sigma_3$ as an automorphism are $D1$, $D2$, $D3$, $D4$, $D7$, $D8$, $D9$, among which $D1$, $D2$ also have the reflection $\sigma_4$. The algebras with the reflection $\sigma_1$ as an automorphism are $D5$, $D10$, $D11$, $D13$.

\subsection{Type A algebras are noetherian.}

Let $D$ be a graded $k$-algebra. Suppose $g_1$ is a normal element of positive degree in $D$. By \cite[Lemma 8.2]{ATV1}, $D$ is noetherian if $D/g_1$ is noetherian. We can repeat the same argument with the algebra $D/(g_1)$, if there is an element $g_2 \in D$, with its image $\overline{g_2} \in D/(g_1)$ a normal element of positive degree in $D/(g_1)$. A sequence of elements $g_1, g_2, \ldots, g_n$ is called a normalizing sequence if $\overline{g_{i+1}}$ is a normal element of positive degree in $D/(g_1, \ldots, g_i)$. An algebra $D$ is said to have enough normal elements, if there is a normalizing sequence $g_1, \ldots, g_n$ in $D$, with $D/(g_1, \ldots, g_n)$ a finite dimensional $k$-vector space. By \cite[Theorem 0.2]{Zh}, a regular algebra with enough normal elements is strongly noetherian and Cohen-Macaulay. While we do not have a uniform method yet, we have constructed enough normal elements for each type A algebra. Thus we have the following result.
\begin{theorem}\label{noetherian}
A type A pushout has enough normal elements, is Auslander-regular, Cohen-Macaulay, and is a strongly noetherian domain.
\end{theorem}
When we list the type A algebras we include a normalizing sequence for each algebra. We note that the choice of normalizing sequence may not be unique.

\subsection{Examples of Clifford and Skew Clifford algebra.}

It can be easily seen that the type A algebras do not have degree one normal elements. So in showing that they have enough normal elements, we start by looking for quadratic normal elements. It turns out many of the type A algebras have a normalizing sequence of four quadratic elements, which makes them good candidates of Clifford and skew Clifford algebras, as recently defined in \cite{CV}.
\begin{definition} \cite[Definition 1.12]{CV}
A {\it graded skew Clifford algebra} \\
$A = A( \mu, M_1, \ldots , M_n)$ associated to $\mu$ and matrices $M_1, \ldots, M_n$ is a graded $k$-algebra on degree-one generators $x_1,\ldots,x_n$ and on degree-two generators $y_1,\ldots,y_n$ with defining relations given by:
\begin{enumerate}
\item $x_ix_j+\mu_{ij}x_jx_i = \sum_{k=1}^n(M_k)_{ij}y_k$ for all $i,j=1,\ldots,n,$;
\item the existence of a normalizing sequence ${r_1,\ldots,r_n}$ that spans $ky_1+\cdots+ky_n$.
\end{enumerate}
\end{definition}

Out of the type A algebras, the algebra $D15$ is a graded Clifford algebra, and the algebras $D2$, $D3$, $D5$, $D6$, $D7$, $D8$, $D10$, $D11$, $D12$, $D13$, $D14$ are graded skew Clifford algebras. The algebra $D4$, with $q^2=-p$, and the algebra $D9$, with either $p^2 = 1$, or the coefficient $b_4=0$ are also skew Clifford algebras. On the other hand, for the algebras $D1$, $D4$ with $q^2= p$, and $D9$ with $p^2=-1$ and $b_4 \ne 0$, we either can not find a normalizing sequence of four quadratic elements, or can not match the normalizing sequence we found with the requirements of skew Clifford algebra. At this stage, we do not have an effective method to show that these three algebras are not skew Clifford algebras. We leave this as a future question.

\subsection{Point Modules of Type A algebras.}

Since type A algebras are Auslander-regular and noetherian, by \cite{VV1, ShV} we can apply Corollary \ref{pmD} to get that their point modules are either given by a pair of compatible point modules, or are the two modules corresponding to $e_3$ and $e_4$. We next give an overview of compatible point modules.

Recall that if $M$ is a point module of $A$, then $M$ can be represented by a sequence of points $(\alpha^0, \alpha^1, \alpha^2, \ldots )$, where each $\alpha^i$ is a point in $\mathbb{P}^2$. By \cite{ATV1}, the first coordinate $\alpha^1$ is given by a cubic divisor $L_A$, and there is an automorphism $\sigma$ of $L_A$ with $\sigma^n(\alpha^0) = \alpha^n$. Similarly the point module of the algebra $B$ is determined by a cubic divisor $L_B$, and an automorphism $\tau$ of $L_B$.

Now express $L_A$ as $(\alpha_1, \alpha_2, \alpha_3)$, we can extend $L_A$ to a surface $(\alpha_1, \alpha_2, \alpha_3, \alpha_4)$ in $\mathbb{P}^3$ by adding a coordinate $\alpha_4$. We call this surface $L_1$. Similarly we extend $\sigma(L_A)$, $L_B$, $\tau(L_B)$ to surfaces $L_2$, $L_3$, $L_4$. Then the compatible point modules are given by the intersection of $(L_1,L_2)$ with $(L_3, L_4)$ inside $\mathbb{P}^3\times\mathbb{P}^3$. Here we have to include the second copy of $\mathbb{P}^3$ to guarantee that the point modules are compatible over the algebra $C$.

Another way to state the above characterization is that compatible point modules are given by points $\alpha \in L_A$, $\beta \in L_B$, such that $\alpha$ and $\beta$ agrees on the first two coordinates, and $\sigma(\alpha)$ and $\tau(\beta)$ also agrees on the first two coordinates.

Using the above method, it is relatively easy to compute the number of point modules for type A algebras. We point out that the relations we use to represent algebra $A$ is not the best for characterizing point modules. For example, the algebra $A04$ contains subclasses of algebras, with different point schemes. Thus to get a better handle on point modules, we should break $A04$ into subtypes. We have a separate study of these algebras, and omit working with them in the context of type A algebras. Thus in this paper we only include the number of compatible point modules for the algebras $D1-D8$, $D10$, $D11$, and $D13$.

We conclude our article by giving a full list of type A algebras. To keep the notation simple, algebras $B$ are isomorphic to algebras $A$ we listed in the previous section, with the coefficients $a_i$ replaced by $b_i$. We also note that any type A algebra with $A04$ as a subalgebra inherits the same generic conditions on coefficients (lemma \ref{generic}), and hence is only generically regular.

\begin{enumerate}
%%%%%%%
\item[D1] $= A21 \cup B21$, with $a_1 = b_1$.
\[ p^2=1, q^2=1 \]
\begin{eqnarray*}
r_1 &=& x_3^2 -  x_2x_1 + x_1x_2 + a_1 x_1^2 \\
r_2 &=& x_3x_1 - px_1x_3 \\
r_3 &=& x_3x_2 - px_2x_3 + p(1-a_1) x_1x_3 \\
r_4 &=& x_4^2 -  x_2x_1 + x_1x_2 + a_1 x_1^2 \\
r_5 &=& x_4x_1 - qx_1x_4 \\
r_6 &=& x_4x_2 - qx_2x_4 + q(1-a_1) x_1x_4 
\end{eqnarray*}
This algebra has a normalizing sequence $x_2x_1-x_1x_2-a_1x_1^2$, $x_4x_3-x_3x_4$, $x_1$, $x_2$. If $p=q$, then there are infinitely many compatible point modules, characterized by two lines and three additional points. If $p<>q$, then there are infinitely many compatible point modules, characterized by two lines and five additional points.

\begin{comment}

\[ \alpha_3(\alpha_3^2 +p(2a_1-1)\alpha_1^2) = 0 \]
\[ \alpha_4(\alpha_4^2 +q(2a_1-1)\alpha_1^2) = 0 \]
\[ (\alpha_1\beta_2-\alpha_2\beta1+a_1\alpha_1\beta_1)(\alpha_1\beta_2-\alpha_2\beta_1+(1-a_1)\alpha_1\beta_1) = 0 \]
\[ \alpha_3\beta_3 - \alpha_2\beta_1 + \alpha_1\beta_2 + a_1\alpha_1\beta_1 = 0\]
\[ \alpha_3\beta_1 - p \alpha_1\beta_3 = 0 \]
\[ \alpha_3\beta_2 - p \alpha_2\beta_3 + p(1-a_1)\alpha_1\beta_3 = 0 \]
\[ \alpha_4\beta_4 - \alpha_2\beta_1 + \alpha_1\beta_2 + a_1\alpha_1\beta_1 = 0\]
\[ \alpha_4\beta_1 - q \alpha_1\beta_4 = 0 \]
\[ \alpha_4\beta_2 - q \alpha_2\beta_4 + q(1-a_1)\alpha_1\beta_4 = 0 \]

\end{comment}

%%%%%%%%
\item[D2] $=A11 \cup B11$.

\[a_{23} \ne 0, \ c_2 \ne 0 \]
\begin{eqnarray*}
r_1 &=& x_3^2  -  x_2x_1 - c_2a_{23}^{-2} x_1x_2 \\
r_2 &=& x_3x_1 - a_{23}x_1x_3 \\
r_3 &=& x_3x_2 - a_{23}^{-1} x_2x_3 \\
r_4 &=& x_4^2  -  x_2x_1 - c_2a_{23}^{-2} x_1x_2 \\
r_5 &=& x_4x_1 - a_{23}x_1x_4 \\
r_6 &=& x_4x_2 - a_{23}^{-1} x_2x_4 
\end{eqnarray*}
This algebra has a normalizing sequence $x_4x_3 - x_3x_4$, $x_2x_1 + c_2a_{23}^{-2} x_1x_2$, $x_2^2$, $x_1^2$. There are infinitely many compatible point modules, characterized by two lines and an additional point.

\begin{comment}
\[ \alpha_3(\alpha_3^2 - (a_{23}+c_2a_{23}^{-3})\alpha_2\alpha_1) = 0  \]
\[ \alpha_4(\alpha_4^2 - (a_{23}+c_2a_{23}^{-3})\alpha_2\alpha_1) = 0  \]
\[ (a_{23}^2-c_2a_{23}^{-2})(\alpha_2\beta_1 + c_2a_{23}^{-2}\alpha_1\beta_2)(\alpha_2\beta_1 - a_{23}^2\alpha_1\beta_2) = 0 \]
\end{comment}

%%%%%%%%%
\item[D3] $= A11 \cup B12$.
\[ a_{23}^4=-1, p^2=-1 \]
\begin{eqnarray*}
r_1 &=& x_3^2  -  x_2x_1 + a_{23}^{-2} x_1x_2 \\
r_2 &=& x_3x_1 - a_{23}x_1x_3 \\
r_3 &=& x_3x_2 - a_{23}^{-1} x_2x_3 \\
r_4 &=& x_4^2  -  x_2^2   \\
r_5 &=& x_4x_1 - p x_1x_4  \\
r_6 &=& x_4x_2 - x_2x_4 - x_1^2 
\end{eqnarray*}
This algebra has a normalizing sequence $x_2^2$, $x_1^2$, $x_4x_3 - x_3x_4$, $x_2x_1 - a_{23}^{-2} x_1x_2$. There are two compatible point modules.

\[ \alpha_3(\alpha_3^2 - 2a_{23} \alpha_2\alpha_1) = 0  \]
\[ \alpha_4^3 - \alpha_4\alpha_2^2 - p\alpha_2\alpha_1^2 = 0   \]
\[ \mathbb{P}^1\times \mathbb{P}^1  \]

%%%%%%%
\item[D4] $=A11 \cup B13$.

\[ p^2=p-1,  \  q^4= p-1, \ b_{11}b_{25} \ne 1-p \]

\begin{eqnarray*}
r_1 &=& x_3^2  -  x_2x_1 - q^2 x_1x_2 \\
r_2 &=& x_3x_1 - qx_1x_3 \\
r_3 &=& x_3x_2 - q^{-1} x_2x_3 \\
r_4 &=& x_4^2 -  x_2x_1 + p x_1x_2  \\
r_5 &=& x_4x_1 - b_{25}x_2^2 - (1-p) x_1x_4 \\
r_6 &=& x_4x_2 - p x_2x_4 - b_{11}x_1^2 
\end{eqnarray*}
If $q^2=p$, then this algebra has a normalizing sequence $x_2^2$, $x_1^2$, $x_3^2$, $x_4^2$, $x_4x_3+x_3x_4$, and eighteen compatible point modules. If $q^2=-p$, then this algebra has a normalizing sequence $x_2x_1+q^2x_1x_2$, $x_4x_3 - q^3 x_3x_4$, $x_2^2$, $x_1^2$, and sixteen compatible point modules.

% \[ \alpha_3(\alpha_3^2 - 2q \alpha_2\alpha_1) = 0  \]
% \[ \alpha_4^3 + (2p-2-b_{11}b_{25})\alpha_4\alpha_2\alpha_1 -pb_{25}\alpha_2^3 + b_{11}\alpha_1^3 = 0  \]
% \[ \mathbb{P}^1 \times \mathbb{P}^1  \]

\begin{comment}
We further reduce the algebra $D4$ to the following two subcases.
\begin{enumerate}
\item
Algebra $D4a$, with $q^2=p$:
\[ q^4 - q^2 + 1 = 0 \]
\begin{eqnarray*}
r_1 &=& x_3^2  -  x_2x_1 - q^2 x_1x_2 \\
r_2 &=& x_3x_1 - qx_1x_3 \\
r_3 &=& x_3x_2 - q^{-1} x_2x_3 \\
r_4 &=& x_4^2 -  x_2x_1 + q^2 x_1x_2  \\
r_5 &=& x_4x_1 - b_{25}x_2^2 - (1-q^2) x_1x_4 \\
r_6 &=& x_4x_2 - q^2 x_2x_4 - b_{11}x_1^2 
\end{eqnarray*}
This algebra has a normalizing sequence $x_2^2$, $x_1^2$, $x_3^2$, $x_4^2$, $x_4x_3+x_3x_4$.

\item
Algebra $D4b$, with $q^2=-p$: 
\[ q^4 + q^2 + 1 = 0 \]
\begin{eqnarray*}
r_1 &=& x_3^2  -  x_2x_1 - q^2 x_1x_2 \\
r_2 &=& x_3x_1 - q x_1x_3 \\
r_3 &=& x_3x_2 - q^{-1} x_2x_3 \\
r_4 &=& x_4^2 -  x_2x_1 - q^2 x_1x_2  \\
r_5 &=& x_4x_1 - b_{25}x_2^2 - (1+q^2) x_1x_4 \\
r_6 &=& x_4x_2 + q^2 x_2x_4 - b_{11}x_1^2
\end{eqnarray*}
This algebra is a skew Clifford algebra, with a normalizing sequence $x_2x_1+q^2x_1x_2$, $x_4x_3 - q^3 x_3x_4$, $x_2^2$, $x_1^2$.
\end{enumerate}
\end{comment}

%%%%%%%
\item[D5] $=A12 \cup B12$

\[ p^2 = q^2 = -1 \]
\begin{eqnarray*}
r_1 &=& x_3^2  -  x_2^2   \\
r_2 &=& x_3x_1 - p x_1x_3  \\
r_3 &=& x_3x_2 - x_2x_3 - x_1^2 \\
r_4 &=& x_4^2  -  x_2^2   \\
r_5 &=& x_4x_1 - q x_1x_4  \\
r_6 &=& x_4x_2 - x_2x_4 - x_1^2 
\end{eqnarray*}
This algebra has a normalizing sequence $x_1^2$, $x_2^2$, $x_4x_3+x_3x_4$, $x_2x_1 - x_1x_2$. There are infinitely many compatible point modules, characterized by three lines and six points.

% \[ \alpha_3^3 - \alpha_3\alpha_2^2 - p\alpha_2\alpha_1^2 = 0   \]
% \[ \alpha_4^3 - \alpha_4\alpha_2^2 - q\alpha_2\alpha_1^2 = 0   \]
% \[ \mathbb{P}^1 \times \mathbb{P}^1  \]

%%%%%%%
\item[D6] $=A13 \cup B13$.

\[ p^2-p+1 =0,  \  a_{11}a_{25} \ne 1-p, \ b_{11}b_{25} \ne 1-p \]
\begin{eqnarray*}
r_1 &=& x_3^2 -  x_2x_1 + p x_1x_2  \\
r_2 &=& x_3x_1 - a_{25}x_2^2 - (1-p) x_1x_3 \\
r_3 &=& x_3x_2 - p x_2x_3 - a_{11}x_1^2 \\
r_4 &=& x_4^2 -  x_2x_1 + p x_1x_2  \\
r_5 &=& x_4x_1 - b_{25}x_2^2 - (1-p) x_1x_4 \\
r_6 &=& x_4x_2 - p x_2x_4 - b_{11}x_1^2 
\end{eqnarray*}
This algebra has a normalizing sequence $x_4x_3+x_3x_4$, $x_2x_1 - p x_1x_2$, $x_2^2$, $x_1^2$. There are eighteen compatible point modules.

% \[ \alpha_3^3 + (2p-2-a_{11}a_{25})\alpha_3\alpha_2\alpha_1 -pa_{25}\alpha_2^3 + a_{11}\alpha_1^3 = 0  \]
% \[ \alpha_4^3 + (2p-2-b_{11}b_{25})\alpha_4\alpha_2\alpha_1 -pb_{25}\alpha_2^3 + b_{11}\alpha_1^3 = 0  \]
% \[ \mathbb{P}^1 \times \mathbb{P}^1  \]

%%%%%%%%%% No need for A01 \cup B01, since A01 is a subcase of A11.
\item[D7] $=A01 \cup B02 $. 
\[ p^4=1 \]
\begin{eqnarray*}
r_1 &=& x_3^2  -  x_2x_1 -  p^2 x_1x_2 \\
r_2 &=& x_3x_1 - p x_1x_3 \\
r_3 &=& x_3x_2 - p^{-1} x_2x_3 \\
r_4 &=& x_4^2 -  x_2^2 \\
r_5 &=& x_4x_1 - x_1x_4 \\
r_6 &=& x_4x_2 + x_2x_4 - b_{15}x_2^2 - x_1^2 
\end{eqnarray*}
This algebra has a normalizing sequence $x_2^2$, $x_1^2$, $x_2x_1 + p^2 x_1x_2$,  $x_4x_3 + px_3x_4$. If $p^2=1$ then there are twelve compatible point modules, and If $p^2=-1$ then there are ten compatible point modules.

% \[ \alpha_3(\alpha_3^2 - 2p\alpha_2\alpha_1) = 0   \]
% \[ \alpha_4^3 - b_{15}\alpha_4^2\alpha_2 + \alpha_3\alpha_2^2 - \alpha_2\alpha_1^2 = 0 \]
% \[ \mathbb{P}^1 \times \mathbb{P}^1  \]

%%%%%%%%%%%%
\item[D8] $=A01 \cup B03 $. 
\[ p^4=1,  b_{11}(b_{15}-b_{11}) \ne 1 \]
\begin{eqnarray*}
r_1 &=& x_3^2  -  x_2x_1 -  p^2 x_1x_2 \\
r_2 &=& x_3x_1 - p x_1x_3 \\
r_3 &=& x_3x_2 - p^{-1} x_2x_3 \\
r_4 &=& x_4^2  - x_2^2 -  x_1^2 \\
r_5 &=& x_4x_1 - x_1x_4 \\
r_6 &=& x_4x_2 + x_2x_4 - b_{15}x_2^2 - b_{11}x_1^2 
\end{eqnarray*}
This algebra has a normalizing sequence $x_1^2$, $x_2^2$, $x_2x_1 + p^2 x_1x_2$, $x_4x_3 + px_3x_4$. There are six compatible point modules.

% \[ \alpha_3(\alpha_3^2 - 2p\alpha_2\alpha_1) = 0   \]
% \[ \alpha_4^3 - b_{15}\alpha_4^2\alpha_2 + \alpha_4\alpha_2^2 - \alpha_4\alpha_1^2 +(a_{15}-a_{11}) \alpha_2\alpha_1^2 = 0   \]
% \[ \mathbb{P}^1 \times \mathbb{P}^1  \]

%%%%%%%%%%%%
\item[D9] $=A01 \cup B04 $. 
\[ p^4=1 \]
\begin{eqnarray*}
r_1 &=& x_3^2  -  x_2x_1 -  p^2 x_1x_2 \\
r_2 &=& x_3x_1 - p x_1x_3 \\
r_3 &=& x_3x_2 - p^3 x_2x_3 \\
r_4 &=& x_4^2  -  b_{5}x_2^2 -  b_{4}x_2x_1 -  b_{4}x_1x_2 -  b_{1}x_1^2 \\
r_5 &=& x_4x_1 - b_{25}x_2^2 + x_1x_4 - b_{21}x_1^2 \\
r_6 &=& x_4x_2 + x_2x_4 - b_{15}x_2^2 - b_{11}x_1^2 
\end{eqnarray*}
This algebra has a normalizing sequence $x_2^2$, $x_1^2$, $x_2x_1 + p^2 x_1x_2$, $x_4x_3 - x_3x_4$. 

% \[ \alpha_3(\alpha_3^2 - 2p\alpha_2\alpha_1) = 0   \]
% \[ \alpha_4^3 - b_{15} \alpha_4^2\alpha_2 + b_5\alpha_3\alpha_2^2 - b_{21}\alpha_4^2\alpha_1 + (2b_4+b_{15}b_{21}-b_{11}b_{25})\alpha_4\alpha_2\alpha_1 + b_1\alpha_4\alpha_1^2 + b_4b_{25}\alpha_2^3 + (b_1b_{25}-b_4b_{15}-b_5b_{21})\alpha_2^2\alpha_1 + (b_5b_{11}-b_1b_{15}-b_4b_{21})\alpha_2\alpha_1^2 + b_4b_{11}\alpha_1^3   \]
% \[ \mathbb{P}^1 \times \mathbb{P}^1  \]

\begin{comment}
If either $p^2=1$ or the coefficient $b_4=0$, then $D9$ is a skew Clifford algebra, with a normalizing sequence $x_2^2$, $x_1^2$, $x_2x_1 + p^2 x_1x_2$, $x_4x_3 - x_3x_4$.

On the other hand, if $p^2=-1$ and $b_4 \ne 0$, then we can re-scale the variables to choose $b_4=1$ and obtain the following algebra:

Algebra $D9a$.
\[ p^2=-1 \]
\begin{eqnarray*}
r_1 &=& x_3^2  -  x_2x_1 + x_1x_2 \\
r_2 &=& x_3x_1 - p x_1x_3 \\
r_3 &=& x_3x_2 + p x_2x_3 \\
r_4 &=& x_4^2  -  b_{5}x_2^2 -  x_2x_1 -  x_1x_2 -  b_{1}x_1^2 \\
r_5 &=& x_4x_1 - b_{25}x_2^2 + x_1x_4 - b_{21}x_1^2 \\
r_6 &=& x_4x_2 + x_2x_4 - b_{15}x_2^2 - b_{11}x_1^2 
\end{eqnarray*}
We have a normalizing sequence $x_2^2$, $x_1^2$, $x_2x_1 - x_1x_2$, $x_4x_3 - x_3x_4$. 
\end{comment}

%%%%%%%%%%%%
\item[D10] $=A02 \cup B02 $. 

\begin{eqnarray*}
r_1 &=& x_3^2 -  x_2^2 \\
r_2 &=& x_3x_1 - x_1x_3 \\
r_3 &=& x_3x_2 + x_2x_3 - a_{15}x_2^2 - x_1^2  \\
r_4 &=& x_4^2 -  x_2^2 \\
r_5 &=& x_4x_1 - x_1x_4 \\
r_6 &=& x_4x_2 + x_2x_4 - b_{15}x_2^2 - x_1^2  
\end{eqnarray*}
This algebra is a skew Clifford algebra, with a normalizing sequence $x_1^2$, $x_2^2$, $x_4x_3 - x_3x_4$, $x_2x_1 - x_1x_2$. There are ten compatible point modules.

% \[ \alpha_3^3 - a_{15}\alpha_3^2\alpha_2 + \alpha_3\alpha_2^2 - \alpha_2\alpha_1^2 = 0 \]
% \[ \alpha_4^3 - b_{15}\alpha_4^2\alpha_2 + \alpha_4\alpha_2^2 - \alpha_2\alpha_1^2 = 0 \]
% \[ \mathbb{P}^1 \times \mathbb{P}^1  \]

%%%%%%%%%%%%
\item[D11] $=A02 \cup B03 $.
\[ b_{11}(b_{15}-b_{11}) \ne 1 \]
\begin{eqnarray*}
r_1 &=& x_3^2 -  x_2^2 \\
r_2 &=& x_3x_1 - x_1x_3 \\
r_3 &=& x_3x_2 + x_2x_3 - a_{15}x_2^2 - x_1^2  \\
r_4 &=& x_4^2  - x_2^2 -  x_1^2 \\
r_5 &=& x_4x_1 - x_1x_4 \\
r_6 &=& x_4x_2 + x_2x_4 - b_{15}x_2^2 - b_{11}x_1^2
\end{eqnarray*}
This algebra is a skew Clifford algebra, with a normalizing sequence $x_1^2$, $x_2^2$, $x_4x_3 - x_3x_4$, $x_2x_1 - x_1x_2$. There are eighteen compatible point modules.

% \[ \alpha_3^3 - a_{15}\alpha_3^2\alpha_2 + \alpha_3\alpha_2^2 - \alpha_2\alpha_1^2 = 0 \]
% \[ \alpha_4^3 - b_{15}\alpha_4^2\alpha_2 + \alpha_4\alpha_2^2 - \alpha_4\alpha_1^2 +(b_{15}-b_{11}) \alpha_2\alpha_1^2 = 0   \]
% \[ \mathbb{P}^1 \times \mathbb{P}^1  \]

%%%%%%%%%%%%
\item[D12] $=A02 \cup B04 $, with generic coefficients.

\begin{eqnarray*}
r_1 &=& x_3^2 -  x_2^2 \\
r_2 &=& x_3x_1 - x_1x_3 \\
r_3 &=& x_3x_2 + x_2x_3 - a_{15}x_2^2 - x_1^2  \\
r_4 &=& x_4^2  -  b_{5}x_2^2 -  b_{4}x_2x_1 -  b_{4}x_1x_2 -  b_{1}x_1^2 \\
r_5 &=& x_4x_1 - b_{25}x_2^2 + x_1x_4 - b_{21}x_1^2 \\
r_6 &=& x_4x_2 + x_2x_4 - b_{15}x_2^2 - b_{11}x_1^2 
\end{eqnarray*}
This algebra is a skew Clifford algebra, with a normalizing sequence $x_2^2$, $x_1^2$, $x_2x_1 + x_1x_2$, $x_4x_3 - x_3x_4$.

% \[ \alpha_3^3 - a_{15}\alpha_3^2\alpha_2 + \alpha_3\alpha_2^2 - \alpha_2\alpha_1^2 = 0 \]
% \[ \alpha_4^3 - b_{15} \alpha_4^2\alpha_2 + b_5\alpha_4\alpha_2^2 - b_{21}\alpha_4^2\alpha_1 + (2b_4+b_{15}b_{21}-b_{11}b_{25})\alpha_4\alpha_2\alpha_1 + b_1\alpha_4\alpha_1^2 + b_4b_{25}\alpha_2^3 + (b_1b_{25}-b_4b_{15}-b_5b_{21})\alpha_2^2\alpha_1 + (b_5b_{11}-b_1b_{15}-b_4b_{21})\alpha_2\alpha_1^2 + b_4b_{11}\alpha_1^3   \]
% \[ \mathbb{P}^1 \times \mathbb{P}^1  \]

%%%%%%%%%%%%
\item[D13] $=A03 \cup B03 $.

\[ a_{11}(a_{15}-a_{11}) \ne 1 \]
\[ b_{11}(b_{15}-b_{11}) \ne 1 \]
\begin{eqnarray*}
r_1 &=& x_3^2  - x_2^2 -  x_1^2 \\
r_2 &=& x_3x_1 - x_1x_3 \\
r_3 &=& x_3x_2 + x_2x_3 - a_{15}x_2^2 - a_{11}x_1^2 \\
r_4 &=& x_4^2  - x_2^2 -  x_1^2 \\
r_5 &=& x_4x_1 - x_1x_4 \\
r_6 &=& x_4x_2 + x_2x_4 - b_{15}x_2^2 - b_{11}x_1^2
\end{eqnarray*}
This algebra is a skew Clifford algebra, with a normalizing sequence $x_1^2$, $x_2^2$, $x_2x_1 + x_1x_2$, $x_4x_3 - x_3x_4$. There are eighteen compatible point modules.

% \[ \alpha_3^3 - a_{15}\alpha_3^2\alpha_2 + \alpha_3\alpha_2^2 - \alpha_3\alpha_1^2 +(a_{15}-a_{11}) \alpha_2\alpha_1^2 = 0   \]
% \[ \alpha_4^3 - b_{15}\alpha_4^2\alpha_2 + \alpha_4\alpha_2^2 - \alpha_4\alpha_1^2 +(b_{15}-b_{11}) \alpha_2\alpha_1^2 = 0   \]
% \[ \mathbb{P}^1 \times \mathbb{P}^1  \]

%%%%%%%%%%%%
\item[D14] $=A03 \cup B04 $, with generic coefficients.

\[ a_{11}(a_{15}-a_{11}) \ne 1 \]
\begin{eqnarray*}
r_1 &=& x_3^2  - x_2^2 -  x_1^2 \\
r_2 &=& x_3x_1 - x_1x_3 \\
r_3 &=& x_3x_2 + x_2x_3 - a_{15}x_2^2 - a_{11}x_1^2 \\
r_4 &=& x_4^2  -  b_{5}x_2^2 -  b_{4}x_2x_1 -  b_{4}x_1x_2 -  b_{1}x_1^2 \\
r_5 &=& x_4x_1 - b_{25}x_2^2 + x_1x_4 - b_{21}x_1^2 \\
r_6 &=& x_4x_2 + x_2x_4 - b_{15}x_2^2 - b_{11}x_1^2 
\end{eqnarray*}
This algebra is a skew Clifford algebra, with a normalizing sequence $x_2^2$, $x_1^2$, $x_2x_1 + x_1x_2$, $x_4x_3 - x_3x_4$.

%\[ \alpha_3^3 - a_{15}\alpha_3^2\alpha_2 + \alpha_3\alpha_2^2 - \alpha_3\alpha_1^2 +(a_{15}-a_{11}) \alpha_2\alpha_1^2 = 0   \]
%\[ \alpha_4^3 - b_{15} \alpha_4^2\alpha_2 + b_5\alpha_4\alpha_2^2 - b_{21}\alpha_4^2\alpha_1 + (2b_4+b_{15}b_{21}-b_{11}b_{25})\alpha_4\alpha_2\alpha_1 + b_1\alpha_4\alpha_1^2 + b_4b_{25}\alpha_2^3 + (b_1b_{25}-b_4b_{15}-b_5b_{21})\alpha_2^2\alpha_1 + (b_5b_{11}-b_1b_{15}-b_4b_{21})\alpha_2\alpha_1^2 + b_4b_{11}\alpha_1^3   \]
%\[ \mathbb{P}^1 \times \mathbb{P}^1  \]

%%%%%%%%%%%%
\item[D15] $=A04 \cup B04 $, with generic coefficients.

\begin{eqnarray*}
r_1 &=& x_3^2  -  a_{5}x_2^2 -  a_{4}x_2x_1 -  a_{4}x_1x_2 -  a_{1}x_1^2 \\
r_2 &=& x_3x_1 - a_{25}x_2^2 + x_1x_3 - a_{21}x_1^2 \\
r_3 &=& x_3x_2 + x_2x_3 - a_{15}x_2^2 - a_{11}x_1^2 \\
r_4 &=& x_4^2  -  b_{5}x_2^2 -  b_{4}x_2x_1 -  b_{4}x_1x_2 -  b_{1}x_1^2 \\
r_5 &=& x_4x_1 - b_{25}x_2^2 + x_1x_4 - b_{21}x_1^2 \\
r_6 &=& x_4x_2 + x_2x_4 - b_{15}x_2^2 - b_{11}x_1^2 
\end{eqnarray*}
This algebra is a Clifford algebra, with a normalizing sequence $x_2^2$,  $x_1^2$, $x_2x_1 + x_1x_2$, $x_4x_3 - x_3x_4$.

% \[ \alpha_3^3 - a_{15} \alpha_3^2\alpha_2 + a_5\alpha_3\alpha_2^2 - a_{21}\alpha_3^2\alpha_1 + (2a_4+a_{15}a_{21}-a_{11}a_{25})\alpha_3\alpha_2\alpha_1 + a_1\alpha_3\alpha_1^2 + a_4a_{25}\alpha_2^3 + (a_1a_{25}-a_4a_{15}-a_5a_{21})\alpha_2^2\alpha_1 + (a_5a_{11}-a_1a_{15}-a_4a_{21})\alpha_2\alpha_1^2 + a_4a_{11}\alpha_1^3   \]
% \[ \alpha_4^3 - b_{15} \alpha_4^2\alpha_2 + b_5\alpha_4\alpha_2^2 - b_{21}\alpha_4^2\alpha_1 + (2b_4+b_{15}b_{21}-b_{11}b_{25})\alpha_4\alpha_2\alpha_1 + b_1\alpha_4\alpha_1^2 + b_4b_{25}\alpha_2^3 + (b_1b_{25}-b_4b_{15}-b_5b_{21})\alpha_2^2\alpha_1 + (b_5b_{11}-b_1b_{15}-b_4b_{21})\alpha_2\alpha_1^2 + b_4b_{11}\alpha_1^3   \]
% \[ \mathbb{P}^1 \times \mathbb{P}^1  \]

\end{enumerate}

\end{document}